%% file: Pie-2025-boundary.tex
\documentclass{amsart}
\usepackage{moreverb,url}
\usepackage{amssymb,amsmath,amsfonts}
\usepackage{bm}
\usepackage[font=scriptsize, labelfont=bf]{caption}
\usepackage{float}
\usepackage{afterpage}
\usepackage{mathrsfs}
\usepackage{subcaption}
\usepackage[squaren]{SIunits}
\usepackage{verbatim}
\usepackage[square,numbers]{natbib}
\usepackage{stmaryrd}
\usepackage{fancyhdr}
\usepackage{syntax,etoolbox}
\usepackage{graphicx}
\usepackage{grffile}
\usepackage{hyperref}
\hypersetup{colorlinks=true,allcolors=blue}
\usepackage[square,numbers]{natbib}
\usepackage[left=1in, right=1in, bottom=1in,top=1in]{geometry}
\usepackage{lineno}
\usepackage{syntax,etoolbox}
\usepackage{algorithmicx}
\usepackage{algorithm}
\usepackage{algpseudocode}
\usepackage{tikz}

\newtheorem{theorem}{Theorem}[section]
\newtheorem{lemma}{Lemma}[section]

\newtheorem{innercustomgeneric}{\customgenericname}
\providecommand{\customgenericname}{}
\newcommand{\newcustomproblem}[2]{%
	\newenvironment{#1}[1]
	{%
		\renewcommand\customgenericname{#2}%
		\renewcommand\theinnercustomgeneric{##1}%
		\innercustomgeneric
	}
	{\endinnercustomgeneric}
}

\newcustomproblem{customprob}{Problem}

\newcommand*{\bqed}{\hfill\ensuremath{\blacksquare}}%

\def\dd{\, \mathrm{d}}

\algblock{Input}{EndInput}

\algnotext{EndInput}

\algblock{Output}{EndOutput}

\algnotext{EndOutput}

\algblock{Iterate}{EndIterate}

\algnotext{EndIterate}

\allowdisplaybreaks

\begin{document}
	
	
	\title[Elliptic membranes in interior normal compliance contact]{On the justification of Koiter's model for elliptic membranes subjected to an interior normal compliance contact condition}
	

	\author[Paolo Piersanti]{Paolo Piersanti}
	\address{School of Science and Engineering, The Chinese University of Hong Kong (Shenzhen), 2001 Longxiang Blvd., Longgang District, Shenzhen, China}
	\email[Corresponding author]{ppiersanti@cuhk.edu.cn}
	
\today

\begin{abstract}
The purpose of this paper is twofold. First, we rigorously justify Koiter's model for linearly elastic elliptic membrane shells in the case where the shell is subject to a geometrical constraint modelled via a normal compliance contact condition defined in the interior of the shell. To achieve this, we establish a novel density result for non-empty, closed, and convex subsets of Lebesgue spaces, which are applicable to cases not covered by the ``density property'' established in [Ciarlet, Mardare \& Piersanti, \emph{Math. Mech. Solids}, 2019].

Second, we demonstrate that the solution to the two-dimensional obstacle problem for linearly elastic elliptic membrane shells, subjected to the interior normal compliance contact condition, exhibits higher regularity throughout its entire definition domain. A key feature of this result is that, while the transverse component of the solution is, in general, only of class $L^2$ and its trace is \emph{a priori} undefined, the methodology proposed here, partially based on [Ciarlet \& Sanchez-Palencia, \emph{J. Math. Pures Appl.}, 1996], enables us to rigorously establish the well-posedness of the trace for the transverse component of the solution by means of an \emph{ad hoc} formula.

\smallskip

\end{abstract}

\maketitle

\input{paper.tex}

\bibliographystyle{abbrvnat} 
\bibliography{references}	

\end{document}

%% file: paper.tex
\section{Introduction}
\label{sec0}

In solid mechanics, shells are defined as three-dimensional structures characterised by a thickness that is significantly smaller than their in-plane dimensions. These structures have garnered considerable attention across Engineering, Physics, and Mathematics due to their remarkable ability to sustain external loads far exceeding what their thin geometry might suggest. This mechanical advantage arises from their intrinsic curvature and the specific boundary conditions they experience. Additionally, their efficiency in material usage, requiring minimal material for construction while offering lightweight and cost-effective solutions, makes them ubiquitous in both natural and industrial contexts. Examples in nature include eggshells, snail shells, turtle carapaces, and blood vessels, while engineered applications span ship hulls, aircraft fuselages, arched roofs, eyeglass lenses, and tyres.

In this paper, we focus on linearly elastic elliptic membrane shells with thickness $2\varepsilon > 0$. Such shells, under appropriate boundary conditions, can sustain applied body forces of order $\mathcal{O}(1)$ relative to $\varepsilon$ (i.e., \emph{independent} of the thickness parameter). This behaviour aligns with established mechanical principles, as elliptic geometries - owing to their curvature-driven stress distribution - are inherently resistant to deformation and buckling.
The first objective of this paper is the justification of Koiter's model for linearly elastic elliptic membrane shells subjected to an obstacle. All the variational problems here considered take the form of a set of variational inequalities posed over a non-empty, closed and convex subset of a suitable space. The constraint associated with the obstacle is formulated in terms of an \emph{interior normal compliance contact condition}, whose formulation depends on the nature of the problem considered.

The interior normal compliance contact condition here considered for modelling the constraint introduced by the presence of the obstacle has the advantage of being applicable to a larger class of surfaces than those for which the ``density property'' introduced by Ciarlet, Mardare \& Piersanti~\cite{CiaMarPie2018b,CiaMarPie2018} applies, and has the advantage of not suffering from the limitations deriving from appearing as a surface integral as originally considered in~\cite{Rodri2018}.

In order to establish the justification of Koiter's model for linearly elastic elliptic membrane shells subjected to the interior normal compliance contact condition, we first depart from the natural three-dimensional formulation of the obstacle problem under consideration, and we study the asymptotic behaviour of the solution as the thickness approaches zero, so as to identify a set of two-dimensional variational inequalities characterised by the arising of the mode of deformation that is dominant in the three-dimensional formulation of the obstacle problem. According to the classical theory of shells, cf., e.g., \cite{Ciarlet2000}, which mode of deformation (stretching or bending) is dominant depends on the \emph{geometry} of the middle surface of the shell, the \emph{boundary conditions} applied on the lateral face of the shell, as well as on the relation between the magnitude of the applied forces (typically, applied body loads and tractions) and the shell thickness.

Second, akin to the classical theory of shells (cf., e.g., \cite{Ciarlet2000}) and the recent papers~\cite{CiaPie2018bCR,CiaPie2018b}, we formulate Koiter's model subjected to the interior normal compliance contact condition and we show that the solution of Koiter's model converges to the solution of the \emph{same} two-dimensional model recovered upon completion of the rigorous asymptotic analysis departing for the three-dimensional formulation of the obstacle problem. By so doing, we can assert that Koiter's model and the original three-dimensional formulation of the obstacle problem asymptotically behave in the same way. The formulation of the interior normal compliance contact condition and the justification of Koiter's model subjected to one such condition constitute the first main novelty in this paper.
The remarkable properties of the asymptotic behaviour of Koiter's model are the starting point for designing convergent numerical schemes for gaining insight into the \emph{average} deformation of the original three-dimensional shell. The main advantage of this approach is that the effect of the locking phenomenon is attenuated as the thickness parameter appears, in Koiter's model, as a multiplicative parameter of the bulk energy terms instead of appearing in the geometry of the definition domain, as it happens for the three-dimensional formulation. The numerical analysis of Koiter's model subjected to the confinement condition was addressed in the case of linearly elastic elliptic membrane shells in~\cite{MeiPie2024,PPS2024-2} by means of the Finite Element Method (cf., e.g., \cite{PGCFEM}). It could be worth investigating, in future works, whether the interior normal compliance contact condition ensures faster convergence.

Paramount to establish the convergence of numerical approximations of solutions via the Finite Element Method is the higher regularity of the solutions for the models under consideration. In the last part of the paper we investigate the higher regularity of the solution of the two-dimensional limit model; this constitutes the second main novelty here proposed. The main difficulty we have to deal with for bringing this item to fruition consists in establishing a trace formula for the transverse component of the displacement, as this function is, \emph{a priori}, only of class $L^2$. We will show, in particular, that the trace formula established by Ciarlet \& Sanchez-Palencia~\cite{CiaSanPan1996} for the transverse component of the displacement of the solution of the two-dimensional limit model holds in the case where the normal compliance contact condition is taken into account~\cite{Pie-2022-interior}. In relation to this, we recall that one such conclusion could not be inferred, in general, for the confinement condition considered in~\cite{CiaPie2018b,CiaPie2018bCR,CiaMarPie2018b,CiaMarPie2018,Pie-2022-interior}, as the geometry of the constraint prevents us from applying the techniques in~\cite{CiaSanPan1996}. Apart from the recent work~\cite{Pie-2022-interior} which was developed in a vectorial framework, for the augmentation of regularity of the solutions of variational inequalities, we recall the main contribution by Frehse~\cite{Frehse1971}, where the sole scalar case was taken into account.

To summarise, the main novelties presented in this paper are the following:
\begin{itemize}
	\item[$(1)$] Formulation of the interior normal compliance contact condition and justification of Koiter's model subjected to one such condition. The main advantage of this approach is that it applies to a larger class of surfaces than those to which the ``density property'' developed in~\cite{CiaMarPie2018} applies, as well as to more general obstacles;
	\item[$(2)$] Augmentation of regularity up to the boundary for the solution of the two-dimensional limit model recovered upon completion of the asymptotic analysis contributing to the justification of Koiter's model. The main difficulty amounts to establishing the trace formula for the transverse component of the solution which is, \emph{a priori}, only of class $L^2$.
\end{itemize}

The paper is divided into seven sections, including this one. In section~\ref{sec1} we recall the background and notation. In section~\ref{sec2}, we provide a formulation of an obstacle problem for a general three-dimensional shell subjected to the interior normal compliance contact condition. In section~\ref{sec3}, we specialise the type of shell under consideration, restricting ourselves to linearly elastic elliptic membrane shells, for which the classical results used throughout the paper will be here recalled. In section~\ref{sec:AA} we conduct a rigorous asymptotic analysis as the thickness tends to zero departing from the \emph{scaled} three-dimensional model formulated in section~\ref{sec3} and subjected to the interior normal compliance contact condition. Upon completion of this section, we will have recovered the two-dimensional limit model characterised by the dominant mode of deformation of linearly elastic elliptic membrane shells, which is the stretching one. In section~\ref{secKoiter} we formulate Koiter's model subjected to the interior normal compliance contact condition, and we show that the solution of one such model converges, in a suitable sense, to the solution of the two-dimensional limit problem recovered upon completion of the asymptotic analysis carried out in section~\ref{sec:AA}. Finally, in section~\ref{sec4}, we give sufficient conditions ensuring that the solution of the two-dimensional limit problem recovered upon completion of section~\ref{sec:AA} and~\ref{secKoiter} is more regular up to the boundary and, moreover, the transverse component of the displacement satisfies the same trace formula as the one established by Ciarlet \& Sanchez-Palencia~\cite{CiaSanPan1996}.

\section{Background and notation}
\label{sec1}

For a comprehensive overview of the classical notions of differential geometry used in this paper, refer to, e.g., \cite{Ciarlet2000} or \cite{Ciarlet2005}.

Greek indices, except $\varepsilon$, take their values in the set $\{1,2\}$, while Latin indices, except when used for ordering sequences, take their values in the set $\{1,2,3\}$. Unless otherwise specified, the summation convention with respect to repeated indices is applied in conjunction with these two rules. 

As a model for the three-dimensional ``physical'' space $\mathbb{R}^3$, we adopt a \emph{real three-dimensional affine Euclidean space}, that is, a set in which a point $O \in \mathbb{R}^3$ is designated as the \emph{origin} and is associated with a \emph{real three-dimensional Euclidean space}, denoted $\mathbb{E}^3$. We equip $\mathbb{E}^3$ with an \emph{orthonormal basis} consisting of three vectors $\bm{e}^i$, with components $e^i_j = \delta^i_j$, where the notation $\delta^j_i$ represents the Kronecker symbol.

The definition of $\mathbb{R}^3$ as an affine Euclidean space implies that to any point $x \in \mathbb{R}^3$ corresponds a uniquely determined vector $\boldsymbol{Ox} \in \mathbb{E}^3$. The origin $O \in \mathbb{R}^3$ and the orthonormal vectors $\bm{e}^i \in \mathbb{E}^3$ together form a \emph{Cartesian frame} in $\mathbb{R}^3$. The three components $x_i$ of the vector $\boldsymbol{Ox}$ with respect to the basis formed by $\bm{e}^i$ are referred to as the \emph{Cartesian coordinates} of $x \in \mathbb{R}^3$, or the \emph{Cartesian components} of $\boldsymbol{Ox} \in \mathbb{E}^3$. Once a Cartesian frame is chosen, any point $x \in \mathbb{R}^3$ can be \emph{identified} with the vector $\boldsymbol{Ox} = x_i \bm{e}^i \in \mathbb{E}^3$. Consequently, a set in $\mathbb{R}^3$ can be identified with a ``physical'' body in the Euclidean space $\mathbb{E}^3$.

The Euclidean inner product and the vector product of $\bm{u}, \bm{v} \in \mathbb{E}^3$ are denoted, respectively, by $\bm{u} \cdot \bm{v}$ and $\bm{u} \wedge \bm{v}$, while the Euclidean norm of $\bm{u} \in \mathbb{E}^3$ is denoted by $\left|\bm{u}\right|$.  The abbreviations ``a.a.'' and ``a.e.'' stand for \emph{almost all} and \emph{almost everywhere}, respectively.

Given an open subset $\Omega$ of $\mathbb{R}^n$, where $n \ge 1$, we denote the usual Lebesgue and Sobolev spaces by $L^2(\Omega)$, $L^1_{\textup{loc}}(\Omega)$, $H^1(\Omega)$, $H^1_0(\Omega)$, $H^1_{\textup{loc}}(\Omega)$. The space of all functions that are infinitely differentiable on $\Omega$ and have compact support in $\Omega$ is denoted by $\mathcal{D}(\Omega)$. We use the notation $\left\| \cdot \right\|_X$ to represent the norm in a normed vector space $X$. Spaces of vector-valued functions are written in boldface.
The Euclidean norm of any point $x \in \Omega$ is denoted by $|x|$.

The boundary $\Gamma$ of an open subset $\Omega$ in $\mathbb{R}^n$ is said to be Lipschitz-continuous if the following conditions are satisfied (cf., e.g., Section~1.18 of \cite{PGCLNFAA}): Given an integer $s \ge 1$, there exist constants $\alpha_1 > 0$ and $L > 0$, and a finite number of local coordinate systems, with coordinates
$$
\bm{\phi}'_r=(\phi_1^r, \dots, \phi_{n-1}^r) \in \mathbb{R}^{n-1} \textup{ and } \phi_r=\phi_n^r, 1 \le r \le s,
$$ 
sets
$$
\tilde{\omega}_r:=\{\bm{\phi}_r \in\mathbb{R}^{n-1}; |\bm{\phi}_r|<\alpha_1\},\quad 1 \le r \le s,
$$
and corresponding functions
$$
\tilde{\theta}_r:\tilde{\omega}_r \to\mathbb{R},\quad 1 \le r \le s,
$$
such that
$$
\Gamma=\bigcup_{r=1}^s \{(\bm{\phi}'_r,\phi_r); \bm{\phi}'_r \in \tilde{\omega}_r \textup{ and }\phi_r=\tilde{\theta}_r(\bm{\phi}'_r)\},
$$
and 
$$
|\tilde{\theta}_r(\bm{\phi}'_r)-\tilde{\theta}_r(\bm{\upsilon}'_r)|\le L |\bm{\phi}'_r-\bm{\upsilon}'_r|, \quad \textup{ for all }\bm{\phi}'_r, \bm{\upsilon}'_r \in \tilde{\omega}_r, \textup{ and all }1\le r\le s.
$$

We observe that the second last formula takes into account overlapping local charts, while the last set of inequalities expresses the Lipschitz continuity of the mappings $\tilde{\theta}_r$.

An open set $\Omega$ is said to be \emph{locally on the same side of its boundary} $\Gamma$ if, in addition, there exists a constant $\alpha_2>0$ such that
\begin{align*}
	\{(\bm{\phi}'_r,\phi_r);\bm{\phi}'_r \in\tilde{\omega}_r \textup{ and }\tilde{\theta}_r(\bm{\phi}'_r) < \phi_r < \tilde{\theta}_r(\bm{\phi}'_r)+\alpha_2\}\subset \Omega&,\quad\textup{ for all } 1\le r\le s,\\
	\{(\bm{\phi}'_r,\phi_r);\bm{\phi}'_r \in\tilde{\omega}_r \textup{ and }\tilde{\theta}_r(\bm{\phi}'_r)-\alpha_2 < \phi_r < \tilde{\theta}_r(\bm{\phi}'_r)\}\subset \mathbb{R}^n\setminus\overline{\Omega}&,\quad\textup{ for all } 1\le r\le s.
\end{align*}

A \emph{domain} in $\mathbb{R}^n$ is a bounded and connected open subset $\Omega$ of $\mathbb{R}^n$, whose boundary $\partial \Omega$ is Lipschitz-continuous, the set $\Omega$ being locally on a single side of $\partial \Omega$.

Let $\omega$ be a domain in $\mathbb{R}^2$ with boundary $\gamma:=\partial\omega$,  and let $\omega_1 \subset \omega$. The special notation $\omega_1 \subset \subset \omega$ means that $\overline{\omega_1} \subset \omega$ and $\textup{dist}(\gamma,\partial\omega_1):=\min\{|x-y|;x \in \gamma \textup{ and } y \in \partial\omega_1\}>0$.
Let $y = (y_\alpha)$ denote a generic point in $\omega$, and let $\partial_\alpha := \partial / \partial y_\alpha$. A mapping $\bm{\theta} \in \mathcal{C}^1(\overline{\omega}; \mathbb{E}^3)$ is said to be an \emph{immersion} if the two vectors
$$
\bm{a}_\alpha (y) := \partial_\alpha \bm{\theta} (y)
$$
are linearly independent at each point $y \in \overline{\omega}$. Then the set $\bm{\theta} (\overline{\omega})$ is a \emph{surface} in $\mathbb{E}^3$, equipped with $y_1, y_2$ as its \emph{curvilinear coordinates}. Given any point $y\in \overline{\omega}$, the linear combinations of the vectors $\bm{a}_\alpha (y)$ span the \emph{tangent plane} to the surface $\bm{\theta} (\overline{\omega})$ at the point $\bm{\theta} (y)$, the unit vector
$$
\bm{a}_3 (y) := \frac{\bm{a}_1(y) \wedge \bm{a}_2 (y)}{|\bm{a}_1(y) \wedge \bm{a}_2 (y)|}
$$
is orthogonal to $\bm{\theta} (\overline{\omega})$ at the point $\bm{\theta} (y)$, the three vectors $\bm{a}_i(y)$ form the \emph{covariant} basis at the point $\bm{\theta} (y)$, and the three vectors $\bm{a}^j(y)$ defined by the relations
$$
\bm{a}^j(y) \cdot \bm{a}_i (y) = \delta^j_i
$$
form the \emph{contravariant} basis at $\bm{\theta} (y)$; note that the vectors $\bm{a}^\beta (y)$ also span the tangent plane to $\bm{\theta} (\overline{\omega})$ at $\bm{\theta} (y)$ and that $\bm{a}^3(y) = \bm{a}_3 (y)$.

The \emph{first fundamental form} of the surface $\bm{\theta} (\overline{\omega})$ is then defined by means of its \emph{covariant components}
$$
a_{\alpha \beta} := \bm{a}_\alpha \cdot \bm{a}_\beta = a_{\beta \alpha} \in \mathcal{C}^0 (\overline{\omega}),
$$
or by means of its \emph{contravariant components}
$$
a^{\alpha \beta}:= \bm{a}^\alpha \cdot \bm{a}^\beta = a^{\beta \alpha}\in \mathcal{C}^0(\overline{\omega}).
$$

Note that the symmetric matrix field $(a^{\alpha \beta})$ is then the inverse of the matrix field $(a_{\alpha \beta})$, that $\bm{a}^\beta = a^{\alpha \beta}\bm{a}_\alpha$ and $\bm{a}_\alpha = a_{\alpha\beta} \bm{a}^\beta$, and that the \emph{area element} along $\bm{\theta}(\overline{\omega})$ is given at each point $\bm{\theta}(y)$, with $y \in \overline{\omega}$, by $\sqrt{a(y)} \dd y$, where
$$
a := \det(a_{\alpha\beta}) \in \mathcal{C}^0(\overline{\omega}).
$$

Given an immersion $\bm{\theta} \in \mathcal{C}^2(\overline{\omega};\mathbb{E}^3)$, the \emph{second fundamental form} of the surface $\bm{\theta}(\overline{\omega})$ is defined by means of its \emph{covariant components}
$$
b_{\alpha\beta}:=\partial_\alpha\bm{a}_\beta\cdot\bm{a}_3 =-\bm{a}_\beta\cdot\partial_\alpha\bm{a}_3=b_{\beta\alpha}\in\mathcal{C}^0(\overline{\omega}),
$$
or by means of its \emph{mixed components}
$$
b^\beta_\alpha:=a^{\beta\sigma} b_{\alpha\sigma} \in \mathcal{C}^0(\overline{\omega}),
$$
and the \emph{Christoffel symbols} associated with the immersion $\bm{\theta}$ are defined by
$$
\Gamma^\sigma_{\alpha \beta}:= \partial_\alpha \bm{a}_\beta \cdot \bm{a}^\sigma = \Gamma^\sigma_{\beta \alpha} \in \mathcal{C}^0 (\overline{\omega}).
$$

The \emph{Gaussian curvature} at each point $\bm{\theta}(y)$, $y \in \overline{\omega}$, of the surface $\bm{\theta}(\overline{\omega})$ is given by:
$$
K(y):=\frac{\det(b_{\alpha\beta} (y))}{\det(a_{\alpha\beta}(y))} = \det\left(b^\beta_\alpha(y)\right).
$$

Note that the denominator in the above relation cannot vanish since $\bm{\theta}$ is assumed to be an immersion. Additionally, observe that the Gaussian curvature $K(y)$ at the point $\bm{\theta}(y)$ is also equal to the product of the two principal curvatures at that point.

Define the \emph{linearised change of metric}, or \emph{strain}, \emph{tensor} associated with the displacement field $\eta_i \bm{a}^i$ by:
$$
\gamma_{\alpha \beta}(\bm{\eta}):= \frac12 (\partial_\beta \eta_\alpha + \partial_\alpha \eta_\beta ) - \Gamma^\sigma_{\alpha \beta} \eta_\sigma - b_{\alpha \beta} \eta_3 = \gamma_{\beta\alpha} (\bm{\eta}).
$$

In this paper, we will consider a specific class of surfaces, as defined below: Let $\omega$ be a domain in $\mathbb{R}^2$. Then, a surface $\bm{\theta}(\overline{\omega})$ defined through an immersion $\bm{\theta} \in \mathcal{C}^2(\overline{\omega};\mathbb{E}^3)$ is referred to as \emph{elliptic} if its Gaussian curvature is strictly positive in $\overline{\omega}$, or equivalently, if there exists a constant $K_0$ such that:
$$
0 < K_0 \leq K(y), \textup{ for all } y \in \overline{\omega}.
$$

It turns out that, when an \emph{elliptic surface} is subjected to a displacement field $\eta_i \bm{a}^i$ whose \emph{tangential covariant components} $\eta_\alpha$ \emph{vanish on the entire boundary of the domain} $\omega$, the following inequality is satisfied. Observe that the components of the displacement fields and the linearised change of metric tensors appearing in the next theorem are no longer assumed to be continuously differentiable functions; instead, they are understood in a generalised sense, as they now belong to \emph{ad hoc} Lebesgue or Sobolev spaces.

\begin{theorem} 
	\label{korn}
	Let $\omega$ be a domain in $\mathbb{R}^2$ and let an immersion $\bm{\theta} \in \mathcal{C}^3(\overline{\omega}; \mathbb{E}^3)$ be given such that the surface $\bm{\theta}(\overline{\omega})$ is elliptic. Define the space
	$$
	\bm{V}_M (\omega) := H^1_0 (\omega) \times H^1_0 (\omega) \times L^2(\omega).
	$$
	
	Then, there exists a constant $c_0=c_0(\omega,\bm{\theta})>0$ such that
	$$
	\left\{\sum_\alpha\|\eta_\alpha\|^2_{H^1(\omega)} + \|\eta_3\|^2_{L^2(\omega)}\right\}^{1/2} \le c_0 \left\{\sum_{\alpha,\beta}\| \gamma_{\alpha\beta}(\bm{\eta})\|_{L^2(\omega)}^2\right\}^{1/2},
	$$
	for all $\bm{\eta}=(\eta_i)\in\bm{V}_M(\omega)$.
	\qed
\end{theorem}

The inequality presented above, derived from \cite{CiaLods1996a} and \cite{CiaSanPan1996} (see also Theorem~2.7-3 in \cite{Ciarlet2000}), serves as an instance of a \emph{Korn's inequality on a surface}. Specifically, it offers an estimate of an appropriate norm of a displacement field defined on a surface in terms of an appropriate norm of a particular ``measure of strain'' (here, the linearised change of metric tensor) associated with the given displacement field.

\section{The three-dimensional obstacle problem for a ``general'' linearly elastic shell} \label{Sec:2}
\label{sec2}

Let $\omega$ be a domain in $\mathbb{R}^2$, let $\gamma := \partial \omega$, and let $\gamma_0$ be a non-empty relatively open subset of $\gamma$. For each $\varepsilon > 0$, we define the sets
$$
\Omega^\varepsilon = \omega \times (-\varepsilon, \varepsilon), \quad \Gamma^\varepsilon_0 := \gamma_0 \times (-\varepsilon, \varepsilon), \quad\textup{and}\quad \Gamma^\varepsilon_{\pm}:=\omega\times\{\pm\varepsilon\},
$$
and we let $x^\varepsilon = (x^\varepsilon_i)$ denote a generic point in the set $\overline{\Omega^\varepsilon}$. We also define $\partial^\varepsilon_i := \partial / \partial x^\varepsilon_i$. Consequently, we have $x^\varepsilon_\alpha = y_\alpha$ and $\partial^\varepsilon_\alpha = \partial_\alpha$.

Given an immersion $\bm{\theta} \in \mathcal{C}^3(\overline{\omega}; \mathbb{E}^3)$ and $\varepsilon > 0$, consider a \emph{shell} with \emph{middle surface} $\bm{\theta}(\overline{\omega})$ and \emph{constant thickness} $2\varepsilon$. This means that the \emph{reference configuration} of the shell is the set $\bm{\Theta}(\overline{\Omega^\varepsilon})$, where the mapping $\bm{\Theta} : \overline{\Omega^\varepsilon} \to \mathbb{E}^3$ is defined by
\begin{equation}
	\label{Theta}
	\bm{\Theta}(x^\varepsilon) := \bm{\theta}(y) + x^\varepsilon_3 \bm{a}^3(y), \quad \textup{at each point } x^\varepsilon = (y, x^\varepsilon_3) \in \overline{\Omega^\varepsilon}.
\end{equation}

One can then show (cf., e.g., Theorem~3.1-1 of \cite{Ciarlet2000}) that, if $\varepsilon > 0$ is sufficiently small, such a mapping $\bm{\Theta} \in \mathcal{C}^2(\overline{\Omega^\varepsilon}; \mathbb{E}^3)$ is an \emph{immersion}, in the sense that the three vectors
$$
\bm{g}^\varepsilon_i(x^\varepsilon) := \partial^\varepsilon_i \bm{\Theta}(x^\varepsilon)
$$
are linearly independent at each point $x^\varepsilon \in \overline{\Omega^\varepsilon}$. These vectors constitute the \emph{covariant basis} at the point $\bm{\Theta}(x^\varepsilon)$, while the three vectors $\bm{g}^{j, \varepsilon}(x^\varepsilon)$, defined by the relations
$$
\bm{g}^{j, \varepsilon}(x^\varepsilon) \cdot \bm{g}^\varepsilon_i(x^\varepsilon) = \delta^j_i,
$$
form the \emph{contravariant basis} at the same point. In the sequel, it will be implicitly assumed that $\varepsilon > 0$ is \emph{small enough} so that $\bm{\Theta} : \overline{\Omega^\varepsilon} \to \mathbb{E}^3$ is an \emph{immersion}.

The \emph{metric tensor associated with the immersion} $\bm{\Theta}$ is defined by means of its \emph{covariant components}
$$
g^\varepsilon_{ij} := \bm{g}^\varepsilon_i \cdot \bm{g}^\varepsilon_j \in \mathcal{C}^1(\overline{\Omega^\varepsilon}),
$$
or by means of its \emph{contravariant components}
$$
g^{ij, \varepsilon} := \bm{g}^{i, \varepsilon} \cdot \bm{g}^{j, \varepsilon} \in \mathcal{C}^1(\overline{\Omega^\varepsilon}).
$$

Note that the symmetric matrix field $(g^{ij, \varepsilon})$ is the inverse of the matrix field $(g^\varepsilon_{ij})$, that $\bm{g}^{j, \varepsilon} = g^{ij, \varepsilon} \bm{g}^\varepsilon_i$ and $\bm{g}^\varepsilon_i = g^\varepsilon_{ij} \bm{g}^{j, \varepsilon}$, and that the \emph{volume element} in $\bm{\Theta}(\overline{\Omega^\varepsilon})$ is given at each point $\bm{\Theta}(x^\varepsilon)$, $x^\varepsilon \in \overline{\Omega^\varepsilon}$, by $\sqrt{g^\varepsilon(x^\varepsilon)} \, \mathrm{d}x^\varepsilon$, where
$$
g^\varepsilon := \det(g^\varepsilon_{ij}) \in \mathcal{C}^1(\overline{\Omega^\varepsilon}).
$$

The \emph{Christoffel symbols} associated with the immersion $\bm{\Theta}$ are defined by
$$
\Gamma^{p, \varepsilon}_{ij} := \partial_i \bm{g}^\varepsilon_j \cdot \bm{g}^{p, \varepsilon} = \Gamma^{p, \varepsilon}_{ji} \in \mathcal{C}^0(\overline{\Omega^\varepsilon}),
$$
and note that $\Gamma^{3, \varepsilon}_{\alpha 3} = \Gamma^{p, \varepsilon}_{33} = 0$.
Given a vector field $\bm{v}^\varepsilon = (v^\varepsilon_i) \in \mathcal{C}^1(\overline{\Omega^\varepsilon}; \mathbb{R}^3)$, the associated vector field
$$
\hat{\bm{v}}^\varepsilon := v^\varepsilon_i \bm{g}^{i, \varepsilon}
$$
can be interpreted as a \emph{displacement field} of the reference configuration $\overline{\hat{\Omega}^\varepsilon}:=\bm{\Theta}(\overline{\Omega^\varepsilon})$ of the shell, defined by means of its \emph{covariant components} $v^\varepsilon_i$ over the vectors $\bm{g}^{i, \varepsilon}$ of the contravariant basis in the reference configuration.

If the norms $\|v^\varepsilon_i\|_{\mathcal{C}^1(\overline{\Omega^\varepsilon})}$ are sufficiently small, the mapping $(\bm{\Theta} + v^\varepsilon_i \bm{g}^{i, \varepsilon})$ is also an immersion. Consequently, the metric tensor of the \emph{deformed configuration} $(\bm{\Theta} + v^\varepsilon_i \bm{g}^{i, \varepsilon})(\overline{\Omega^\varepsilon})$ can be defined by means of its covariant components
\begin{equation*}
	g^\varepsilon_{ij}(\bm{v}^\varepsilon) := (\bm{g}^\varepsilon_i + \partial^\varepsilon_i \tilde{\bm{v}}^\varepsilon) \cdot (\bm{g}^\varepsilon_j + \partial^\varepsilon_j \tilde{\bm{v}}^\varepsilon) \\
	= g^\varepsilon_{ij} + \bm{g}^\varepsilon_i \cdot \partial^\varepsilon_j \tilde{\bm{v}}^\varepsilon + \partial^\varepsilon_i \tilde{\bm{v}}^\varepsilon \cdot \bm{g}^\varepsilon_j + \partial^\varepsilon_i \tilde{\bm{v}}^\varepsilon \cdot \partial^\varepsilon_j \tilde{\bm{v}}^\varepsilon.
\end{equation*}

The linear part with respect to $\tilde{\bm{v}}^\varepsilon$ in the difference $(g^\varepsilon_{ij}(\bm{v}^\varepsilon) - g^\varepsilon_{ij}) / 2$ is called the \emph{linearised strain tensor} associated with the displacement field $v^\varepsilon_i \bm{g}^{i, \varepsilon}$, whose covariant components are defined by:
$$
e^\varepsilon_{i \| j}(\bm{v}^\varepsilon) := \frac{1}{2} \left( \bm{g}^\varepsilon_i \cdot \partial_j^\varepsilon \tilde{\bm{v}}^\varepsilon + \partial^\varepsilon_i \tilde{\bm{v}}^\varepsilon \cdot \bm{g}^\varepsilon_j \right) = \frac{1}{2} (\partial^\varepsilon_j v^\varepsilon_i + \partial^\varepsilon_i v^\varepsilon_j) - \Gamma^{p, \varepsilon}_{ij} v^\varepsilon_p = e_{j \| i}^\varepsilon(\bm{v}^\varepsilon).
$$

The functions $e^\varepsilon_{i\|j}(\bm{v}^\varepsilon)$ are referred to as the \emph{linearised strains in curvilinear coordinates} associated with the displacement field $v^\varepsilon_i \bm{g}^{i, \varepsilon}$.

Throughout this paper, we assume that, for each $\varepsilon > 0$, the reference configuration $\bm{\Theta}(\overline{\Omega^\varepsilon})$ of the shell is a \emph{natural state} (i.e., stress-free) and that the material constituting the shell is \emph{homogeneous}, \emph{isotropic}, and \emph{linearly elastic}. The behaviour of such an elastic material is governed by its two \emph{Lamé constants} $\lambda \ge 0$ and $\mu > 0$ (for details, see, e.g., Section~3.8 of \cite{Ciarlet1988}).

We also assume that the shell is subjected to \emph{applied body forces} with density per unit volume defined by their covariant components $f^{i,\varepsilon} \in L^2(\Omega^\varepsilon)$, and to a \emph{homogeneous boundary condition of place} along the portion $\Gamma^\varepsilon_0$ of its lateral face (i.e., the displacement vanishes on $\Gamma^\varepsilon_0$).

In this paper, we consider a specific \emph{obstacle problem} for such a shell, in the sense that the shell is also subjected to an \emph{interior normal compliance contact condition} similar to the one suggested in~\cite{Rodri2018}. Before introducing this new condition, let us recall how the normal compliance contact conditions was formulated in~\cite{Rodri2018}.
The normal compliance contact condition requires that the magnitude of the displacement of the shell along the normal to the points of the undeformed reference configuration is less than or equal to the gap along the normal to the points of the undeformed reference configuration between the shell and the obstacle. In this framework, the obstacle is not necessarily a plane, as it was instead considered in~\cite{CiaPie2018b,CiaPie2018bCR,CiaMarPie2018b,CiaMarPie2018,Pie-2022-interior,Pie-2022-jde,Pie2023}. However, the normal compliance contact condition comes with the limitation that we are \emph{expected} to know which part of the boundary of the shell engages contact with the obstacle. In the framework of linearised elasticity, where the displacements are \emph{infinitesimal}, this contact condition appears to be sound.

Denote by $\hat{\bm{\nu}}^\varepsilon$ the outer unit normal vector field to the boundary of $\hat{\Omega}^\varepsilon$. Denote by $\hat{s}^\varepsilon:\partial\hat{\Omega}^\varepsilon \to \mathbb{R}^+\cup \{0\}$ the \emph{gap function} measuring the distance between a point $\hat{x}^\varepsilon \in \partial\hat{\Omega}^\varepsilon$ and the obstacle along the vector $\hat{\bm{\nu}}^\varepsilon$. Since the immersion $\bm{\theta}$ is of class $\mathcal{C}^3(\overline{\omega};\mathbb{E}^3)$, the gap function $\hat{s}^\varepsilon$ is of class $\mathcal{C}^2(\partial\hat{\Omega}^\varepsilon)$.

If the linearly elastic shell under consideration is subjected to the action of a displacement field $\hat{\bm{v}}^\varepsilon$, we require that:
\begin{equation}
	\label{NCC-Cart}
	\hat{\bm{v}}^\varepsilon(\hat{x}^\varepsilon) \cdot \hat{\bm{\nu}}^\varepsilon(\hat{x}^\varepsilon)\le \hat{s}^\varepsilon(\hat{x}^\varepsilon),\quad\textup{ for all (or, possibly, a.a.) } \hat{x}^\varepsilon \in \partial\hat{\Omega}^\varepsilon.
\end{equation}

The latter condition is actually stronger than what we need, as we typically need to prescribe it only on either the upper or lower face of the shell, i.e., one between the sets $\Gamma^\varepsilon_+$ and $\Gamma^\varepsilon_-$, respectively. 
The previous definition of gap function can be recast in curvilinear coordinates, by defining the mapping $s^\varepsilon:\partial\Omega^\varepsilon \to \mathbb{R}^+\cup\{0\}$ in a way such that:
\begin{equation*}
	s^\varepsilon(x^\varepsilon):=\hat{s}^\varepsilon(\hat{x}^\varepsilon),\quad\textup{for all }x^\varepsilon\in\partial\Omega^\varepsilon,
\end{equation*}
and observe that this mapping is also of class $\mathcal{C}^2(\partial\Omega^\varepsilon)$. Since $\hat{\bm{v}}^\varepsilon=v_i^\varepsilon \bm{g}^{i,\varepsilon}$, and since $\hat{\bm{\nu}}^\varepsilon=-\bm{g}^{3,\varepsilon}$, the normal compliance contact condition~\eqref{NCC-Cart} in curvilinear coordinates takes the following form:
\begin{equation}
	\label{NCC-Curv}
	-v_3^\varepsilon(x^\varepsilon) \le s^\varepsilon(x^\varepsilon),\quad\textup{ for all (or, possibly, a.a.) }x^\varepsilon \in \Gamma^\varepsilon_- \cup \Gamma^\varepsilon_+.
\end{equation}

In the same fashion, we denote the gap function associated with the middle surface $\bm{\theta}(\overline{\omega})$ by $\hat{s}:\bm{\theta}(\overline{\omega}) \to \mathbb{R}^+\cup \{0\}$. This function measures the distance between the middle surface $\bm{\theta}(\overline{\omega})$ and the obstacle along the normal $\bm{a}^3(y)$ to a point $\bm{\theta}(y)$ of the middle surface, for all $y\in\overline{\omega}$.

Let us observe that, in light of~\eqref{Theta}, the points in the undeformed reference configuration of the shell are obtained by moving along the lines normal to $\bm{\theta}(\overline{\omega})$. This allows us to extend the gap function $s^\varepsilon$ to the entire set $\overline{\Omega^\varepsilon}$ in the following fashion:
\begin{equation*}
	s^\varepsilon(x^\varepsilon):=s(y)+x_3^\varepsilon,\quad\textup{for all } x^\varepsilon=(y,x_3^\varepsilon)\in\overline{\Omega^\varepsilon}.
\end{equation*}

Clearly, the same operation can be performed in terms of $\hat{s}^\varepsilon$ and $\hat{s}$ although the relation expressed in curvilinear coordinates is the one we will be resorting to in the remainder of this paper. With this in mind, starting from~\eqref{NCC-Curv}, we are in position to introduce the \emph{interior normal compliance contact condition}, which amounts to requiring:
\begin{equation}
	\label{INCC-Curv}
	-v_3^\varepsilon(x^\varepsilon) \le s^\varepsilon(x^\varepsilon),\quad\textup{ for all (or, possibly, a.a.) }x^\varepsilon \in \Omega^\varepsilon.
\end{equation}

Under the interior normal compliance contact condition, at every point on the undeformed reference configuration of the shell, the normal component of the shell displacement evaluated at the points of the undeformed configuration must not exceed the initial gap between the shell and the obstacle, measured along the normal to the undeformed shell at that point.
The interior normal compliance contact condition is less general than the confinement condition considered in~\cite{CiaPie2018b,CiaPie2018bCR,CiaMarPie2018b,CiaMarPie2018,Pie-2022-interior,Pie-2022-jde,Pie2020-1,Pie2023}, which requires that \emph{all} the points of the deformed reference configuration do not cross an obstacle made of a finite number of planes. While the confinement condition appears to be suitable for non-linear obstacle problems as well, the interior normal compliance contact condition seems to be more suitable for obstacle problems characterised by \emph{small} deformations, for which it can be forecast which part of the boundary will engage contact with the obstacle.

We also observe that the interior normal compliance contact condition~\eqref{INCC-Curv} is stronger than the normal compliance contact condition considered in~\cite{Rodri2018}, where only the points of the shell that were likely to engage contact with the obstacle were required to satisfy the constraint.
Notice that, in light of the absolute continuity of Sobolev functions along segments parallel to the axes of the orthonormal system under consideration, if a function satisfies the interior normal compliance contact condition, then it satisfies the normal compliance contact condition considered in~\cite{Rodri2018}. Even though the contact condition we are considering in this paper is stronger than the one considered in~\cite{Rodri2018}, we observe first that it is reasonable from the physical point, as it applies to all the points of the deformed reference configuration, and second that is the key for justifying Koiter's model for a larger class of surfaces than those to which the ``density property'' originally formulated in~\cite{CiaMarPie2018b,CiaMarPie2018} applies.

The mathematical models characterised by the interior normal compliance contact condition taken into account here do not account for any traction forces. Indeed, classical mechanics dictates that no traction forces can be applied to the portion of the three-dimensional shell boundary in contact with the obstacle. Friction is not considered in this analysis.

The mathematical modelling of such an \emph{obstacle problem for a linearly elastic shell} is now clear. Apart from the confinement condition, the rest, including the \emph{function space} and the expression of the quadratic \emph{energy} $J^\varepsilon$, follows classical formulations (cf.~\cite{Ciarlet2000}). Specifically, let
$$
A^{ijk\ell,\varepsilon} := \lambda g^{ij,\varepsilon} g^{k\ell,\varepsilon} + \mu \left(g^{ik, \varepsilon} g^{j\ell,\varepsilon} + g^{i\ell,\varepsilon} g^{jk,\varepsilon}\right) = A^{jik\ell,\varepsilon} = A^{k\ell ij,\varepsilon},
$$
denote the contravariant components of the \emph{elasticity tensor} of the elastic material constituting the shell. The unknown of the problem is the vector field $\bm{u}^\varepsilon = (u^\varepsilon_i)$, where the functions $u^\varepsilon_i:\overline{\Omega^\varepsilon} \to \mathbb{R}$ are the three covariant components of the unknown ``three-dimensional'' displacement vector field $u^\varepsilon_i \bm{g}^{i,\varepsilon}$ of the reference configuration of the shell. The function $\bm{u}^\varepsilon$ minimises the quadratic energy $J^\varepsilon:\bm{H}^1(\Omega^\varepsilon) \to \mathbb{R}$ defined by
$$
J^\varepsilon(\bm{v}^\varepsilon):= \dfrac{1}{2} \int_{\Omega^\varepsilon} A^{ijk\ell,\varepsilon} e^\varepsilon_{k\|\ell}(\bm{v}^\varepsilon) e^\varepsilon_{i\|j}(\bm{v}^\varepsilon) \sqrt{g^\varepsilon} \dd x^\varepsilon - \int_{\Omega^\varepsilon} f^{i,\varepsilon} v^\varepsilon_i \sqrt{g^\varepsilon} \dd x^\varepsilon,
$$
for each $\bm{v}^\varepsilon = (v^\varepsilon_i) \in \bm{H}^1(\Omega^\varepsilon)$, over the \emph{set of admissible displacements} defined by:
\begin{equation*}
	\bm{U}(\Omega^\varepsilon) := \{\bm{v}^\varepsilon = (v^\varepsilon_i) \in \bm{H}^1(\Omega^\varepsilon); \, \bm{v}^\varepsilon = \bm{0} \textup{ on } \Gamma^\varepsilon_0 \textup{ and } v_3^\varepsilon+s^\varepsilon \ge 0 \textup{ a.e. in } \Omega^\varepsilon\}.
\end{equation*}

The solution to this \emph{minimisation problem} exists, is unique, and can also be characterised as the solution of a set of appropriate variational inequalities (cf.~\cite{Rodri2018}).

\begin{customprob}{$\mathcal{P}(\Omega^\varepsilon)$}\label{problem0}
	Find $\bm{u}^\varepsilon \in \bm{U} (\Omega^\varepsilon)$ that satisfies the following variational inequalities:
	$$
	\int_{\Omega^\varepsilon} 
	A^{ijk\ell, \varepsilon} e^\varepsilon_{k\| \ell}  (\bm{u}^\varepsilon)
	\left( e^\varepsilon_{i\| j}  (\bm{v}^\varepsilon) -  e^\varepsilon_{i\| j}  (\bm{u}^\varepsilon)  \right) \sqrt{g^\varepsilon} \dd x^\varepsilon \geq \int_{\Omega^\varepsilon} f^{i , \varepsilon} (v^\varepsilon_i - u^\varepsilon_i)\sqrt{g^\varepsilon} \dd x^\varepsilon,
	$$
	for all $\bm{v}^\varepsilon = (v^\varepsilon_i) \in \bm{U}(\Omega^\varepsilon)$.
	\bqed	
\end{customprob}

The following result derives from Korn's inequality (Theorem~1.7-4 in~\cite{Ciarlet2000}) and classical arguments (cf., e.g., Theorem~6.1-2 in~\cite{PGCLNFAA}).

\begin{theorem} \label{t:2}
	The quadratic minimisation problem: Find a vector field $\bm{u}^\varepsilon \in \bm{U}(\Omega^\varepsilon)$ such that
	$$
	J^\varepsilon (\bm{u}^\varepsilon) = \inf_{\bm{v}^\varepsilon \in \bm{U} (\Omega^\varepsilon)} J^\varepsilon (\bm{v}^\varepsilon)
	$$
	has one and only one solution. Besides, the vector field $\bm{u}^\varepsilon$ is also the unique solution of Problem~\ref{problem0}.
	\qed
\end{theorem}

\section{The scaled three-dimensional problem for a family of linearly elastic elliptic membrane shells} \label{sec3}

In Section~\ref{Sec:2}, we analysed an obstacle problem for ``general'' linearly elastic shells. From now on, we will focus on a specific class of shells, as defined below, which was originally introduced in \cite{Ciarlet1996} (see also \cite{Ciarlet2000}).

Consider a linearly elastic shell, subject to the assumptions outlined in Section~\ref{Sec:2}. Such a shell is referred to as a \emph{linearly elastic elliptic membrane shell} (hereafter simply called a \emph{membrane shell}) if the following two additional conditions are satisfied: \emph{first}, $\gamma_0 = \gamma$, meaning that the homogeneous boundary condition of place is imposed over the \emph{entire lateral face} $\gamma \times (-\varepsilon, \varepsilon)$ of the shell; and \emph{second}, its middle surface $\bm{\theta}(\overline{\omega})$ is \emph{elliptic}, as defined in Section~\ref{sec1}.

In this paper, we study the \emph{obstacle problem} (as defined in Section~\ref{sec2}) for a family of membrane shells, all sharing the \emph{same middle surface}, with the thickness $2\varepsilon > 0$ considered as a ``small'' parameter tending to zero. To perform a rigorous asymptotic analysis of the three-dimensional model as $\varepsilon \to 0$, we adopt a (by now standard) methodology first introduced in \cite{CiaDes1979}. This approach involves ``scaling'' the solution of Problem~\ref{problem0}, for $\varepsilon > 0$, over a \emph{fixed domain} $\Omega$, using appropriate \emph{scalings of the unknowns} and \emph{assumptions on the data}.

More specifically, let
$$
\Omega := \omega \times (-1, 1),
$$
let $x = (x_i)$ denote a generic point in the set $\overline{\Omega}$, and let $\partial_i := \partial / \partial x_i$. We denote by $\Gamma_0$ the lateral boundary of $\Omega$, namely the set $\Gamma_0:=\gamma\times(-1,1)$. We also define the sets $\Gamma_{\pm}:=\omega\times\{\pm 1\}$. For each point $x = (x_i) \in \overline{\Omega}$, we define the point $x^\varepsilon = (x^\varepsilon_i)$ by
$$
x^\varepsilon_\alpha := x_\alpha = y_\alpha \quad \textup{and} \quad x^\varepsilon_3 := \varepsilon x_3,
$$
so that $\partial^\varepsilon_\alpha = \partial_\alpha$ and $\partial^\varepsilon_3 = \varepsilon^{-1} \partial_3$. To the unknown $\bm{u}^\varepsilon = (u^\varepsilon_i)$ and the vector fields $\bm{v}^\varepsilon = (v^\varepsilon_i)$ appearing in the formulation of Problem~\ref{problem0} for a membrane shell, we associate the \emph{scaled unknown} $\bm{u}(\varepsilon) = (u_i(\varepsilon))$ and the \emph{scaled vector fields} $\bm{v} = (v_i)$ by defining
$$
u_i(\varepsilon)(x) := u^\varepsilon_i(x^\varepsilon) \quad \textup{and} \quad v_i(x) := v^\varepsilon_i(x^\varepsilon)
$$
for all (or, possibly, a.a.) $x \in \overline{\Omega}$.

With the gap function $s^\varepsilon$, we associate (the following constitutes another assumption on the data) the corresponding \emph{scaled gap function} $s(\varepsilon):\overline{\Omega}\to\mathbb{R}^+ \cup \{0\}$ independent of $\varepsilon$, and defined in a way that
\begin{equation*}
	s(\varepsilon)(x):=s^\varepsilon(x^\varepsilon),\quad\textup{for all }x\in\overline{\Omega},
\end{equation*}
or, equivalently:
\begin{equation*}
	s(\varepsilon)(x)=s(y)+\varepsilon x_3,\quad\textup{for all }x=(y,x_3)\in\overline{\Omega}.
\end{equation*}

Finally, we assume that there exist functions $f^i \in L^2(\Omega)$, independent of $\varepsilon$, such that the following \emph{assumptions on the data} hold:
\begin{equation}
	\label{ass-data}
	f^{i, \varepsilon}(x^\varepsilon) = f^i(x), \quad \textup{for all } x \in \Omega.
\end{equation}

Note that the independence of the Lamé constants from $\varepsilon$, assumed in Section~\ref{Sec:2} in the formulation of Problem~\ref{problem0}, implicitly constitutes another \emph{assumption on the data}.

The variational problem $\mathcal{P}(\varepsilon; \Omega)$, defined next, will serve as the starting point for the asymptotic analysis presented in \cite{CiaMarPie2018}. Define the scaled counterpart of the functions and vector fields introduced in the formulation of Problem~\ref{problem0} as follows:
\begin{equation*}
	\begin{aligned}
		\bm{g}^i(\varepsilon)(x) &:= \bm{g}^{i,\varepsilon}(x^\varepsilon),\quad \textup{ at each } x\in \overline{\Omega},\\
		g(\varepsilon)(x) &:= g^\varepsilon(x^\varepsilon),\\
		A^{ijk\ell}(\varepsilon)(x) &:= A^{ijk\ell,\varepsilon}(x^\varepsilon),\quad \textup{ at each } x\in \overline{\Omega}, \\
		\Gamma^p_{ij}(\varepsilon)(x) &:= \Gamma^{p,\varepsilon}_{ij}(x^\varepsilon),\quad \textup{ at each } x\in \overline{\Omega},\\
		e_{\alpha\|\beta}(\varepsilon;\bm{v}) &:= \frac{1}{2}(\partial_\beta v_\alpha + \partial_\alpha v_\beta) - \Gamma^k_{\alpha\beta}(\varepsilon) v_k = e_{\beta\|\alpha}(\varepsilon;\bm{v}) , \\
		e_{3\|\alpha}(\varepsilon; \bm{v}) &:=\dfrac{1}{2} \left(\dfrac{1}{\varepsilon} \partial_3 v_\alpha + \partial_\alpha v_3\right) - \Gamma^\sigma_{\alpha3}(\varepsilon) v_\sigma=e_{\alpha\|3}(\varepsilon;\bm{v}),\\
		e_{3\|3}(\varepsilon;\bm{v}) &:= \dfrac{1}{\varepsilon} \partial_3 v_3.
	\end{aligned}
\end{equation*}

The next lemma assembles various asymptotic properties as $\varepsilon \to 0$ of functions and vector fields appearing in the formulation of Problem~\ref{problem1} (cf., e.g., Theorems~3.3-1 and~3.3-2 in~\cite{Ciarlet2000}); these properties will be repeatedly used in the proof of the convergence theorem (Theorem~\ref{th1}).

In the next statement, the notation ``$\mathcal{O}(\varepsilon)$'', or ``$\mathcal{O}(\varepsilon^2)$'', stands for a remainder that is of order $\varepsilon$, or $\varepsilon^2$, with respect to the sup-norm over the set $\overline{\Omega}$, and any function, or vector-valued function, of the variable $y \in \overline{\omega}$, such as $a^{\alpha \beta}, b_{\alpha \beta}, \bm{a}^i$, etc. is identified with the function, or vector-valued function, of $x =(y,x_3) \in \overline{\Omega} = \overline{\omega} \times \left[-1, 1\right]$ that takes the same value at $x_3 = 0$ and is independent of $x_3 \in \left[-1, 1\right]$; for brevity, this extension from $\overline{\omega} $ to $\overline{\Omega}$ is designated with the same notation. Recall that $\varepsilon > 0$ is implicitly assumed to be small enough so that $\bm{\Theta} : \overline{\Omega^\varepsilon} \to \mathbb{E}^3$ is an injective immersion.

\begin{lemma}
	\label{lem:2}
	Let $\varepsilon_0$ be defined as in Theorem~3.1-1 of~\cite{Ciarlet2000}. The functions $A^{ijk\ell}(\varepsilon) = A^{jik\ell}(\varepsilon) = A^{k\ell ij}(\varepsilon)$ have the following properties:
	\begin{align*}
		A^{ijk\ell}(\varepsilon)&= A^{ijk\ell}(0) + \mathcal{O}(\varepsilon),\\
		A^{\alpha\beta\sigma 3}(\varepsilon) &= A^{\alpha 333}(\varepsilon)=0,
	\end{align*}
	for all $0<\varepsilon \le \varepsilon_0$, where
	\begin{align*}
		A^{\alpha\beta\sigma\tau}(0) &= \lambda a^{\alpha\beta} a^{\sigma\tau} + \mu (a^{\alpha\sigma}a^{\beta\tau} + a^{\alpha\tau} a^{\beta\sigma}),\\
		A^{\alpha \beta 33} (0) &= \lambda a^{\alpha\beta},\\
		A^{\alpha 3 \sigma 3}(0) &= \mu a^{\alpha\sigma},\\
		A^{3333}(0) &= \lambda + 2\mu,
	\end{align*}
	and there exists a constant $C_e>0$ such that
	$$
	\sum_{i,j}| t_{ij}|^2 \le C_e A^{ijk\ell}(\varepsilon)(x) t_{k\ell} t_{ij},
	$$
	for all $0<\varepsilon \le \varepsilon_0$, all $x \in \overline{\Omega}$, and all symmetric matrices $\{t_{ij}\}_{i,j}$.
	
	The functions $\Gamma^p_{ij}(\varepsilon)$ and $g(\varepsilon)$ have the following properties:
	\begin{align*}
		\Gamma^\sigma_{\alpha\beta}(\varepsilon) &= \Gamma^\sigma_{\alpha\beta}-\varepsilon x_3(\partial_\alpha b^\sigma_\beta + \Gamma^\sigma_{\alpha\tau} b^\tau_\beta - \Gamma^\tau_{\alpha\beta} b^\sigma_\tau) + \mathcal{O}(\varepsilon^2), \\
		\Gamma^3_{\alpha\beta}(\varepsilon) &= b_{\alpha \beta} - \varepsilon x_3 b^\sigma_\alpha b_{\sigma\beta},\\
		\partial_3 \Gamma^p_{\alpha\beta}(\varepsilon)&=\mathcal{O}(\varepsilon),\\
		\Gamma^\sigma_{\alpha 3}(\varepsilon) &= -b^\sigma_\alpha -\varepsilon x_3 b^\tau_\alpha b^\sigma_\tau + \mathcal{O}(\varepsilon^2),\\
		\Gamma^3_{\alpha 3}(\varepsilon)&=\Gamma^p_{33}(\varepsilon )=0,\\
		g(\varepsilon)&=a+\mathcal{O}(\varepsilon),
	\end{align*}
	for all $0<\varepsilon \le \varepsilon_0$ and all $x \in \overline{\Omega}$. In particular then, there exist constants $g_0$ and $g_1$ such that
	$$
	0 < g_0 \le g(\varepsilon)(x) \le g_1 \textup{ for all }0<\varepsilon \le \varepsilon_0 \textup{ and all }x \in \overline{\Omega}.
	$$
	
	The vector fields $\bm{g}_i (\varepsilon)$ and $\bm{g}^j(\varepsilon)$ have the following properties:
	\begin{align*}
		\bm{g}_\alpha(\varepsilon)&=\bm{a}_\alpha-\varepsilon x_3 b^\sigma_\alpha \bm{a}_\sigma,\\
		\bm{g}_3 (\varepsilon)&=\bm{a}_3 , \\
		\bm{g}^\alpha(\varepsilon)&=\bm{a}^\alpha + \varepsilon x_3 b^\alpha_\sigma\bm{a}^\sigma+\mathcal{O}(\varepsilon^2),\\
		\bm{g}^3(\varepsilon)&=\bm{a}^3.
	\end{align*}
	\qed
\end{lemma}

Define the space
\begin{equation*}
	\bm{V}(\Omega):=\{\bm{v} = (v_i) \in \bm{H}^1(\Omega); \bm{v} = \bm{0} \textup{ on } \gamma \times (-1,1) \},
\end{equation*}
and, for each $\varepsilon > 0$, define the set
\begin{equation*}
	\bm{U}(\varepsilon;\Omega) := \{\bm{v} = (v_i) \in \bm{V}(\Omega); s(\varepsilon)+v_3\ge 0 \textup{ a.e. in } \Omega\}.
\end{equation*}

The \emph{scaled version} of Problem~\ref{problem0} is formulated as follows.

\begin{customprob}{$\mathcal{P}(\varepsilon;\Omega)$}\label{problem1}
	Find $\bm{u}(\varepsilon) \in \bm{U}(\varepsilon; \Omega)$ that satisfies the variational inequalities
	\begin{equation*}
		\int_\Omega A^{ijk\ell}(\varepsilon) e_{k\| \ell}(\varepsilon; \bm{u}(\varepsilon)) \left(e_{i\| j}(\varepsilon; \bm{v}) - e_{i\|j}(\varepsilon; \bm{u}(\varepsilon))\right) \sqrt{g(\varepsilon)} \dd x 
		\ge \int_\Omega f^i (v_i - u_i(\varepsilon)) \sqrt{g(\varepsilon)} \dd x,
	\end{equation*}
	for all $\bm{v} \in \bm{U}(\varepsilon;\Omega)$.
	\bqed
\end{customprob}

Since Problem~\ref{problem1} is just a re-writing of Problem~\ref{problem0}, it is clear that Problem~\ref{problem1} admits a unique solution (cf., e.g., \cite{CiaMarPie2018}).

Alternatively, one could prove this result independently of the prior knowledge of Problem~\ref{problem0}. One such proof hinges on the following \emph{Korn's inequality in curvilinear coordinates} formulated over the \emph{fixed} domain $\Omega = \omega \times (-1,1)$, according to the following theorem. That the constant $C_1$ that appears in this inequality is \emph{independent of} $\varepsilon > 0$ will play a key role in the derivation of the estimates needed to establish the augmentation of regularity for Problem~\ref{problem2}. For a proof of this inequality, we refer to Theorem~4.1 in~\cite{CiaLods1996b} or Theorem~4.3-1 in~\cite{Ciarlet2000}.

\begin{theorem}\label{korn3D}
	Let there be given a family of linearly elastic elliptic membrane shells with the same middle surface $\bm{\theta}(\overline{\omega})$ and thickness $2 \varepsilon > 0$. Define the space
	$$
	\bm{V}(\Omega) := \{\bm{v} = (v_i) \in \bm{H}^1(\Omega) ; \; \bm{v} = \bm{0} \textup{ on }  \gamma  \times (-1,1)\}.
	$$
	
	Then there exist constants $\varepsilon_1 > 0$ and $C_1 > 0$ such that
	$$
	\left\{\sum_\alpha \left\| v_\alpha \right\|^2_{H^1(\Omega)} + \left\| v_3\right\|^2_{L^2(\Omega)}\right\}^{1/2} \le C_1 \left\{ \sum_{i,j} \left\| e_{i\|j} (\varepsilon; \bm{v}) \right\|^2_{L^2(\Omega)} \right\}^{1/2},
	$$
	for all $0 < \varepsilon \leq \varepsilon_1$ and all $\bm{v} \in \bm{V}(\Omega)$.
	\qed
\end{theorem}

Define the space 
$$
\bm{V}_M (\omega) := H^1_0 (\omega) \times H^1_0 (\omega) \times L^2(\omega),
$$
and define the set:
\begin{equation*}
	\bm{U}_M(\omega) := \{\bm{\eta} = (\eta_i) \in \bm{V}_M(\omega); \eta_3+s \ge 0 \textup{ a.e. in } \omega\}.
\end{equation*}

Define $p^i := \int^1_{-1} f^i \dd x_3$ and define the fourth order two-dimensional elasticity tensor $\{a^{\alpha\beta\sigma\tau}\}$ by:
\begin{equation*}
	a^{\alpha\beta\sigma\tau} := \dfrac{4\lambda\mu}{\lambda + 2\mu} a^{\alpha\beta} a^{\sigma\tau} + 2\mu \left(a^{\alpha\sigma} a^{\beta\tau} + a^{\alpha\tau} a^{\beta\sigma}\right).
\end{equation*}

The two-dimensional limit problem describing the deformation of a linearly elastic elliptic membrane shell subjected to remaining confined in a prescribed half-space takes the following form. This two-dimensional limit problem was recovered upon completion of a rigorous asymptotic analysis as $\varepsilon\to0$ departing from Problem~\ref{problem1}.

\begin{customprob}{$\mathcal{P}_M(\omega)$}\label{problem2}
	Find $\bm{\zeta} \in \bm{U}_M(\omega)$ that satisfies the variational inequalities
	\begin{equation*}
		\int_\omega a^{\alpha \beta \sigma \tau} \gamma_{\sigma\tau}(\bm{\zeta}) \gamma_{\alpha\beta} (\bm{\eta} - \bm{\zeta}) \sqrt{a} \dd y \geq \int_\omega p^i (\eta_i - \zeta_i) \sqrt{a} \dd y,
	\end{equation*}
	for all $\bm{\eta} = (\eta_i) \in \bm{U}_M(\omega)$.
	\bqed
\end{customprob}

Notice that Problem~\ref{problem2} admits a unique solution thanks to the uniform positive-definiteness of the fourth order two-dimensional elasticity tensor $\{a^{\alpha\beta\sigma\tau}\}$ (cf., e.g., Theorem~3.3-2 in~\cite{Ciarlet2000}), Korn's inequality for elliptic surfaces (Theorem~\ref{korn}). Note that, at this stage, we are not allowed to discuss the trace of $\zeta_3$ along $\gamma$, since $\zeta_3$ is \emph{a priori} only of class $L^2(\omega)$. The loss of the homogeneous boundary condition for the transverse component of the limit model, which is \emph{a priori} only square integrable, is \emph{compensated} by the emergence of a \emph{boundary layer} for the transverse component. By proving that the solution possesses higher regularity, we will demonstrate that it is possible to \emph{restore} the boundary condition for the transverse component of the solution as well, and that the trace of the transverse component of the solution along the boundary is almost everywhere (in the sense of the measure of the contour) equal to zero. We will establish one such augmentation of regularity as well as the well-posedness for the concept of trace for the transverse component of the solution in Section~\ref{sec4}.

\section{Asymptotic analysis as $\varepsilon\to 0$ departing from Problem~\ref{problem1}}
\label{sec:AA}

The purpose of this section is to show that the solution $\bm{u}(\varepsilon)$ of Problem~\ref{problem1} converges, in some suitable sense, to the solution of Problem~\ref{problem2} as $\varepsilon\to 0$. This result was established in~\cite{Rodri2018} in the case where the normal compliance contact is formulated over the boundary. In the formulation there proposed, however, the normal compliance contact term is defined over the set $\Gamma_C$ which represents the portion of the lower or upper face which is \emph{likely to engage contact with the obstacle} and which, in general, is one of the \emph{unknowns} of the problem.

The interior normal compliance contact condition here proposed combines the formulation proposed in~\cite{Rodri2018} with the formulation of the confinement condition proposed by Ciarlet and his collaborators~\cite{CiaPie2018b,CiaPie2018bCR,CiaMarPie2018b,CiaMarPie2018,Pie-2022-interior,Pie-2022-jde,Pie2020-1,Pie2023}. By so doing, we will be able to establish the validity of the forthcoming asymptotic analysis for a larger number of cases than those for which the ``density property'' established in~\cite{CiaMarPie2018} applies, and we will be able to give a sound formulation of the model \emph{without} making the unknown set $\Gamma_C$ considered in~\cite{Rodri2018} appear in our analysis.

\begin{lemma}
	\label{lem5}
	Let $\varepsilon_1$ be as in Theorem~\ref{korn3D}.
	Suppose that there exists $L_0>0$ independent of $\varepsilon$ and $\varepsilon_1>0$ such that $s(\varepsilon)\ge L_0$ in $\overline{\Omega}$ for all $0<\varepsilon\le \varepsilon_1$. Then, it results:
	\begin{equation*}
		s \ge \dfrac{L_0}{2},\quad\textup{ in }\overline{\omega}.
	\end{equation*}
\end{lemma}
\begin{proof}
	A direct computation gives $s(y) \ge L_0-\varepsilon x_3 \ge L_0-\varepsilon_1$, for a.a. $x=(y,x_3) \in\Omega$. The conclusion follows up to shrinking $\varepsilon_1$.
\end{proof}

In the same spirit as~\cite{CiaMarPie2018}, we establish a density result for carrying out the forthcoming asymptotic analysis which, differently from the ``density property'' established in~\cite{CiaMarPie2018}, only requires that the undeformed reference configuration of the shell does not engage contact with the obstacle.
This allows us to describe a wider range of situations including the one where the middle surface of the elliptic membrane under consideration is a half-sphere. The latter case was an example of middle surface for which the ``density property'' established in~\cite{CiaMarPie2018} does not apply.

\begin{lemma}
	\label{lem6}
	Assume that $\min_{y\in\overline{\omega}}s>0$.
	Then, the set $\bm{U}_M^{(2)}(\omega):=\{\bm{\eta}=(\eta_i) \in H^1_0(\omega) \times H^1_0(\omega)\times \mathcal{D}(\omega);\eta_3+s\ge 0 \textup{ in }\overline{\omega}\}$ is dense in $\bm{U}_M(\omega)$ with respect to the norm $\|\cdot\|_{H^1(\omega)\times H^1(\omega)\times L^2(\omega)}$.
\end{lemma}
\begin{proof}
	We follow the same strategy as in~\cite{CiaMarPie2018}, showing that the sought density is achieved by proving it gradually for smaller and smaller subsets of $\bm{U}_M(\omega)$. For the sake of clarity, we break the proof into two steps, numbered $(i)$ and $(ii)$.
	
	$(i)$ \emph{The set $\bm{U}_M^{(1)}(\omega):=\{\bm{\eta}=(\eta_i)\in H^1_0(\omega) \times H^1_0(\omega) \times L^2(\omega); \textup{ there exists }\delta^\sharp=\delta^\sharp(\bm{\eta})>0 \textup{ for which } \eta_3(y)=0 \textup{ for a.a. } y\in \omega; \textup{dist}(y,\gamma)<\delta^\sharp \textup{ and }\eta_3+s\ge (\min_{y\in\overline{\omega}}s)\delta^\sharp \textup{ a.e. in }\omega\}$ is dense in $\bm{U}_M(\omega)$.} Fix $\bm{\eta} \in \bm{U}_M(\omega)$. For each integer $k\ge 1$, define the function:
	\begin{equation*}
		f^{(k)}(y):=
		\begin{cases}
			0&, \textup{ if }0\le \textup{dist}(y,\gamma)<\dfrac{1}{k},\\
			\\
			k\textup{dist}(y,\gamma)-1&, \textup{ if }\dfrac{1}{k}\le\textup{dist}(y,\gamma)<\dfrac{2}{k},\\
			\\
			1&, \textup{ if }\textup{dist}(y,\gamma)\ge\dfrac{2}{k}.
		\end{cases}
	\end{equation*}
	
	Observe that $0\le f^{(k)}\le 1$ in $\overline{\omega}$. Define the functions $\eta_\alpha^{(k)}:=\eta_\alpha \in H^1_0(\omega)$ and define:
	$$
	\eta_3^{(k)}:=\left(1-\dfrac{1}{k}\right)f^{(k)}\eta_3 \in L^2(\omega).
	$$
	
	It is straightforward to observe that $\eta_3^{(k)} \to \eta_3$ in $L^2(\omega)$ as $k\to\infty$, and that $\eta_3^{(k)}(y)=0$ for a.a. $y\in\omega$ such that $\textup{dist}(y,\gamma)<1/k$. Besides, for a.a. $y\in\omega$, it results:
	\begin{equation*}
		\begin{aligned}
			\eta_3^{(k)}(y)+s(y)&=\dfrac{s(y)}{k}+\left(1-\dfrac{1}{k}\right)s(y)-\left(1-\dfrac{1}{k}\right)f^{(k)}(y) s(y)+\left(1-\dfrac{1}{k}\right)f^{(k)}(y)(s(y)+\eta_3(y))\\
			&\ge \dfrac{s(y)}{k}+\left(1-\dfrac{1}{k}\right)(1-f^{(k)}(y))s(y)\ge \dfrac{\min_{y\in\overline{\omega}}s(y)}{k}.
		\end{aligned}
	\end{equation*}
	
	Letting $\delta^\sharp:=1/k$ allows to infer the desired conclusion.
	
	$(ii)$ \emph{The set $\bm{U}_M^{(2)}(\omega):=\{\bm{\eta}=(\eta_i) \in H^1_0(\omega) \times H^1_0(\omega)\times \mathcal{D}(\omega);\eta_3+s > 0 \textup{ in }\overline{\omega}\}$ is dense in $\bm{U}_M^{(1)}(\omega)$ with respect to the norm $\|\cdot\|_{H^1(\omega)\times H^1(\omega)\times L^2(\omega)}$.} Fix $\bm{\eta} \in\bm{U}_M^{(1)}(\omega)$. For each integer $k\ge 2/\delta^\sharp(\bm{\eta})$, denote by $\rho_{1/k}$ the standard mollifier with compact support in the ball $B(0;1/k)$. For each $r>0$, we denote by $\omega_r$ the set:
	\begin{equation*}
		\omega_r:=\{y\in\omega;\textup{dist}(y,\gamma)>r\}.
	\end{equation*}
	
	Define $\eta_\alpha^{(k)}:=\eta_\alpha$ and $\eta_3^{(k)}:=\rho_{1/k} \star \eta_3$, where $\star$ denotes the standard convolution product (cf., e.g., \cite{Brez11}). Therefore, it results:
	\begin{equation*}
		\begin{aligned}
			\eta_3^{(k)}=0&,\quad\textup{ in }\omega\setminus\overline{\omega_{\delta^\sharp(\bm{\eta})/2}},\\
			\eta_3^{(k)}\in\mathcal{D}(\omega)&,\quad\textup{ for all } k\ge 1,\\
			\eta_3^{(k)}\to\eta_3&,\quad\textup{ in }L^2(\omega) \textup{ as }k\to\infty.
		\end{aligned}
	\end{equation*}
	
	Additionally, for all $y\in\overline{\omega}$, an application of the fact that $\int_{B(0;1/k)}\rho_{1/k}(z)\dd z =1$ gives:
	\begin{equation*}
		\begin{aligned}
			&s(y)+\eta_3^{(k)}(y)=s(y)+\int_{B(0;1/k)} \rho_{1/k}(z) \eta_3(y-z) \dd z\\
			&= \int_{B(0;1/k)}(s(y)-s(y-z)) \rho_{1/k}(z) \dd z+\int_{B(0;1/k)} (s(y-z)+\eta_3(y-z)) \rho_{1/k}(z) \dd z\\
			&\ge \delta^\sharp(\bm{\eta}) \left(\min_{y\in\overline{\omega}}s(y)\right)-\max_{\substack{y\in\omega_{\delta^\sharp(\bm{\eta})/2}\\|z|<1/k}}|s(y)-s(y-z)|.
		\end{aligned}
	\end{equation*}
	
	The continuity of the gap function $s$ gives that the latter quantity is strictly positive for $k$ sufficiently large, and the proof is complete.
\end{proof}

The convergence result we are going to establish in this section constitutes the first new result in this paper. In order to establish this result, we first need to demonstrate a preliminary density lemma, which plays the same role as the ``density property'' invoked in~\cite{CiaMarPie2018}. In this regard, we observe that in~\cite{Rodri2018}, the test functions were specialised by taking $\bm{v}=-(0,0,-s)$, thus requiring the \emph{additional} assumption that $s$ vanishes along $\Gamma_0$. We present here an improved version of this result, where the latter assumption on $s$ is replaced by the more natural and more general assumption $s(\varepsilon)\ge 0$ in $\overline{\Omega}$.

\begin{theorem}
	\label{th1}
	Let $\omega$ be a domain in $\mathbb{R}^2$, let $\bm{\theta} \in \mathcal{C}^3(\overline{\omega}; \mathbb{E}^3)$ be the middle surface of an elliptic membrane.
	Let there be given a family of elliptic membranes with the same middle surface $\bm{\theta}(\overline{\omega})$ and thickness $2 \varepsilon > 0$, and let $\bm{u}(\varepsilon) \in \bm{U}(\varepsilon;\Omega)$ denote for each $\varepsilon > 0$ the unique solution of Problem~\ref{problem1}. Assume that the applied body force densities $f^{i,\varepsilon}$ are in the form~\eqref{ass-data} and that $s(\varepsilon)> 0$ in $\overline{\Omega}$.
	
	Then, there exists $\bm{u}=(u_i)$ independent of the variable $x_3$ and satisfying
	\begin{align*}
		u_\alpha =0&,\quad \textup{ on }  \Gamma_0=\gamma\times(-1,1),\\
		u_\alpha(\varepsilon) \to u_\alpha&,\quad \textup{ in } H^1(\Omega) \textup{ as } \varepsilon \to 0,\\
		u_3(\varepsilon) \to u_3&,\quad \textup{ in } L^2(\Omega) \textup{ as } \varepsilon \to 0.
	\end{align*}
	
	Define the average
	$$
	\overline{\bm{u}} = (\overline{u}_i) := \frac12 \int^1_{-1} \bm{u} \dd x_3.
	$$
	Then
	$$
	\overline{\bm{u}} = \bm{\zeta},
	$$
	where $\bm{\zeta}$ is the unique solution of Problem~\ref{problem2}.
\end{theorem}
\begin{proof}
	In what follows, strong and weak convergences are respectively denoted by $\to$ and $\rightharpoonup$. The terms $e_{i\|j}(\varepsilon;\bm{u}(\varepsilon))$ appear, in the context of this proof, in the abbreviated form $e_{i\|j}(\varepsilon)$.
	The outline of the proof closely follows the one in~\cite{CiaMarPie2018}, which is itself an adaptation of the techniques presented in~\cite{CiaLods1996b}. For the sake of clarity, the proof is broken into seven parts, numbered~$(i)$--$(vii)$.
	
	$(i)$ \emph{There exists a subsequence, still denoted $\{\bm{u}(\varepsilon)\}_{\varepsilon>0}$, and there exist functions $u_\alpha \in H^1(\Omega)$ satisfying}
	$$
	u_\alpha = 0,\quad \textup{ on } \gamma \times (-1,1),
	$$
	\emph{and functions $u_3 \in L^2(\Omega)$ and $e_{i\|j} \in L^2(\Omega)$ such that}
	\begin{align*}
		u_\alpha(\varepsilon)\rightharpoonup u_\alpha&,\quad \textup{ in } H^1(\Omega) \textup{ as }\varepsilon\to 0,\\
		u(\varepsilon) \to  u_\alpha&,\quad \textup{ in } L^2(\Omega) \textup{ as }\varepsilon\to 0,\\
		u_3 (\varepsilon) \rightharpoonup u_3&,\quad \textup{ in } L^2(\Omega) \textup{ as }\varepsilon\to 0,\\
		e_{i\|j}(\varepsilon) \rightharpoonup e_{i\|j}&,\quad \textup{ in } L^2(\Omega) \textup{ as }\varepsilon\to 0.
	\end{align*}
	
	\emph{Besides, the average $\overline{\bm{u}}$ satisfies $\overline{\bm{u}} \in \bm{U}_M(\omega)$.}
	
	Since $\bm{0} \in \bm{U}(\varepsilon;\Omega)$ by assumption, we can specialise $\bm{v}=\bm{0}$ in the variational inequalities of Problem~\ref{problem1}. Combining the uniform positive-definiteness of the tensor $\{A^{ijk\ell}(\varepsilon)\}$, Korn's inequality (Theorem~\ref{korn3D}), and the asymptotic behaviour of the function $g(\varepsilon)$ (Lemma~\ref{lem:2}), we obtain that the following estimates hold for each $\varepsilon>0$:
	\begin{align*}
		&\dfrac{1}{C_1^2}\sum_i \|u_i(\varepsilon)\|^2_{L^2(\Omega)} \le \dfrac{1}{C_1^2}\left(\sum_\alpha\|u_\alpha(\varepsilon)\|^2_{H^1(\Omega)} + |u_3(\varepsilon)\|^2_{L^2(\Omega)}\right)
		\le \sum_{i,j} \|e_{i\|j}(\varepsilon)\|^2_{L^2(\Omega)}\\
		&\le\frac{C_e}{\sqrt{g_0}} \int_\Omega A^{ijk\ell}(\varepsilon) e_{k\|\ell}(\varepsilon) e_{i\|j}(\varepsilon) \sqrt{g(\varepsilon)} \dd x
		\le \frac{C_e}{\sqrt{g_0}} \int_\Omega f^i u_i(\varepsilon) \sqrt{g(\varepsilon)} \dd x\\
		&\le C_e\sqrt{\frac{g_1}{g_0}} \left\{\sum_i\|f^i\|^2_{L^2(\Omega)}\right\}^{1/2} \left\{\sum_i \left\|u_i(\varepsilon)\right\|^2_{L^2(\Omega)}\right\}^{1/2}.
	\end{align*}
	
	Hence, up to passing to a suitable subsequence, there exists a vector field $\bm{u} \in \bm{L}^2(\Omega)$ and functions $e_{i\|j} \in L^2(\Omega)$ such that:
	\begin{equation}
		\label{conv-proc}
		\begin{aligned}
			u_\alpha(\varepsilon) \rightharpoonup u_\alpha&,\quad \textup{ in } H^1(\Omega) \textup{ as }\varepsilon\to 0,\\
			u_3(\varepsilon) \rightharpoonup u_3&,\quad \textup{ in } L^2(\Omega) \textup{ as }\varepsilon\to 0,\\
			e_{i\|j}(\varepsilon) \rightharpoonup e_{i\|j}&,\quad \textup{ in } L^2(\Omega) \textup{ as }\varepsilon\to 0.
		\end{aligned}
	\end{equation}
	
	The fact that $u_\alpha(\varepsilon) \to u_\alpha \textup{ in } L^2(\Omega)$ is a consequence of the Rellich-Kondra\v{s}ov Theorem (viz., e.g., Theorem~6.6-3 of~\cite{PGCLNFAA}).
	Notice that the formula $s(\varepsilon)=s+\varepsilon x_3$ in turn gives that $s+\overline{u_3(\varepsilon)} \ge 0$ a.e. in $\omega$. An application of Theorem~4.2-1(b) thus gives $\overline{\bm{u}(\varepsilon)} \in\bm{U}_M(\omega)$. Observe that the convexity and closure of the set $\bm{U}_M(\omega)$ imply that the set $\bm{U}_M(\omega)$ is \emph{weakly closed} (cf., e.g., \cite{Brez11}).
	Combining the weak closure of the set $\bm{U}_M(\omega)$ with Theorem~4.2-1(c) in~\cite{Ciarlet2000} and the first two weak convergences in~\eqref{conv-proc} gives that $\overline{\bm{u}} \in\bm{U}_M(\omega)$.
	
	$(ii)$ \emph{The weak limits $u_i$ found in} $(i)$ \emph{are independent of the variable $x_3 \in (-1,1)$, in the sense that they satisfy, respectively,
		$$
		\partial_3 u_\alpha = 0 \textup{ in } L^2(\Omega) \quad \textup{ and } \quad \partial_3 u_3 = 0 \textup{ in } \mathcal{D}'(\Omega).
		$$
	}
	
	The proof is identical to that of part $(ii)$ of the proof of Theorem~4.4-1 in~\cite{Ciarlet2000} and is for this reason omitted.
	
	$(iii)$ \emph{For each} $\bm{\varphi}=(\varphi_i) \in\bm{\mathcal{D}}(\Omega)$ \emph{there exists a vector field} $\bm{v}(\bm{\varphi})=(v_i(\bm{\varphi})) \in \bm{U}(\varepsilon;\Omega)$ \emph{such that} $\|\bm{v}(\bm{\varphi})\|_{\bm{H}^1(\Omega)} \le C(\bm{\varphi})$, \emph{for some} $C(\bm{\varphi})>0$ \emph{depending on} $\bm{\varphi}$, \emph{and such that} $\partial_3 v_i(\bm{\varphi})=\varphi_i$ \emph{a.e. in} $\Omega$ for all $1\le i \le 3$.
	
	Fix $\bm{\varphi}\in\bm{\mathcal{D}}(\Omega)$. Define $\delta:=\textup{dist}(\Gamma_0,\textup{supp }\bm{\varphi})$, and observe that $\delta>0$ since the support of $\bm{\varphi}$ is compact in $\Omega$ by assumption. For each $r>0$, denote by $\omega_r$ the set:
	$$
	\omega_r:=\{y\in\omega;\textup{dist}(y,\gamma)>r\}.
	$$
	
	Denote by $\rho_{\delta/3}$ the standard mollifier with support contained in the ball $B(0;\delta/3) \subset \omega$ and denote by $\star$ the standard convolution product (cf., e.g., \cite{Brez11}).
	For each $1\le i \le 3$, define:
	\begin{equation}
		\label{v}
		v_i(\bm{\varphi})(x):=\left(\max_{x\in\overline{\Omega}}|\varphi_i(x)|\right) (\rho_{\delta/3}\star\chi_{\omega_{\delta/2}})(y)+\int_{0}^{x_3} \varphi_i(y,z) \dd z, \quad\textup{ for a.a. }x=(y,x_3)\in\Omega.
	\end{equation}
	
	Observe that $\bm{v}(\bm{\varphi}) \in \bm{V}(\Omega)$ and that $\partial_3 v_i(\bm{\varphi})=\varphi_i$, a.e. in $\Omega$, for all $1\le i \le 3$.
	Since $\rho_{\delta/3}\star\chi_{\omega_{\delta/2}}=1$ a.e. in $\textup{supp }\bm{\varphi}$, we obtain that~\eqref{v} gives:
	\begin{equation*}
		v_3(\bm{\varphi})(x)+s(\varepsilon)(x)\ge s(\varepsilon)(x)> 0,\quad\textup{ for a.a. }x\in\Omega,
	\end{equation*}
	thus showing that $\bm{v}(\bm{\varphi})$ verifies the constraint. An instance of the constant $C(\bm{\varphi})>0$ for which the estimate holds is $C(\bm{\varphi})=2\left(\max_{1\le i \le 3}\max_{x\in\overline{\Omega}}|\varphi_i(x)|\right)+\|\bm{\varphi}\|_{\bm{H}^1(\Omega)}$.
	
	$(iv)$ \emph{The weak limits $e_{i\|j} \in L^2(\Omega)$, $u_\alpha \in H^1(\Omega)$ and $u_3 \in L^2(\Omega)$ found in} $(i)$ \emph{satisfy}
	\begin{align*}
		e_{\alpha\|\beta}&=\gamma_{\alpha\beta}(\bm{u}) \textup{ in }L^2(\Omega),\\
		e_{\alpha\|3}&=0,\\
		e_{3\|3}&=-\frac{\lambda}{\lambda + 2\mu} a^{\alpha\beta} e_{\alpha\|\beta} \textup{ in } \Omega.
	\end{align*}
	
	Given any vector field $\bm{\varphi} = (\varphi_i) \in \boldsymbol{\mathcal{D}}(\Omega)$ and $\varepsilon_1$ as in Theorem~\ref{korn3D}, let the vector fields $\bm{v}(\bm{\varphi}) \in \bm{U}(\varepsilon;\Omega)$ be defined as in part (iii). Then the proof is identical to that of part~(v) in Theorem~4.1 of~\cite{CiaMarPie2018}.
	
	$(v)$ \emph{Let $\varepsilon_1$ be as in Theorem~\ref{korn3D}. Given any $\bm{\eta} \in \bm{U}_M(\omega) \cap \bm{H}^1_0(\omega)$, one can find a vector field $\tilde{\bm{v}}(\varepsilon;\bm{\eta})\in\bm{U}(\varepsilon;\Omega)$ such that:}
	\begin{equation*}
		\begin{aligned}
			\tilde{\bm{v}}(\varepsilon;\bm{\eta}) \in\bm{U}(\varepsilon;\Omega)&, \quad \textup{ for all } 0 < \varepsilon \le \varepsilon_1,\\
			\tilde{\bm{v}}(\varepsilon;\bm{\eta})\to \bm{\eta}&, \quad \textup{ in } \bm{H}^1(\Omega) \textup{ as } \varepsilon \to 0,\\
			e_{\alpha\|\beta}(\varepsilon;\tilde{\bm{v}}(\varepsilon;\bm{\eta})) \to \gamma_{\alpha\beta}(\bm{\eta})&, \quad \textup{ in } L^2(\Omega)\textup{ as } \varepsilon \to 0, \\
			e_{\alpha\|3}(\varepsilon;\tilde{\bm{v}}(\varepsilon;\bm{\eta}))\to\dfrac{\partial_\alpha\eta_3}{2}-b_\alpha^\sigma \eta_\sigma&,\quad\textup{ in } L^2(\Omega) \textup{ as }\varepsilon\to 0,\\
			e_{3\|3}(\varepsilon;\tilde{\bm{v}}(\varepsilon;\bm{\eta})) = 0&,\quad \textup{ a.e. in }\Omega.
		\end{aligned}
	\end{equation*}
	
	Given $\bm{\eta}\in\bm{U}_M(\omega)\cap\bm{H}^1_0(\omega)$, we want to show that that the vector field $\tilde{\bm{v}}(\varepsilon;\bm{\eta})$ whose components are defined by $\tilde{v}_i(\varepsilon,\bm{\eta}):=(1-\sqrt{\varepsilon})\eta_i$ satisfies the announced properties.
	That $\bm{v}(\varepsilon;\bm{\eta})\in\bm{H}^1(\Omega)$ and $\tilde{\bm{v}}(\varepsilon;\bm{\eta}) \to \bm{\eta}$ in $\bm{H}^1(\Omega)$ as $\varepsilon\to 0$ are obvious facts. Additionally, observe that $s(\varepsilon)+(1-\sqrt{\varepsilon})\eta_3\ge \sqrt{\varepsilon}s(\varepsilon)>0$, a.e. in $\Omega$, thus showing that, in fact, it is true that $\tilde{\bm{v}}(\varepsilon;\bm{\eta}) \in \bm{U}(\varepsilon;\Omega)$.
	The other properties are verified in light of the definition of the components of the linearised strain tensor $e_{i\|j}(\varepsilon,\cdot)$ and the definition of the components of the change of metric tensor $\gamma_{\alpha\beta}$.
	
	$(vi)$ \emph{The weak limit} $\bm{u}$ \emph{recovered in part} $(i)$ \emph{is such that} $\overline{\bm{u}}$ \emph{is the unique solution of Problem~\ref{problem2}.} That $\overline{\bm{u}} \in \bm{U}_M(\omega)$ has already been established in part~$(ii)$. In correspondence of any $\bm{\eta}\in\bm{U}_M(\omega)\cap \bm{H}^1_0(\omega)$, we can construct a vector field $\tilde{\bm{v}}(\varepsilon;\bm{\eta})$ as in part~$(v)$ and we can test the variational inequalities in Problem~\ref{problem1} at this vector field.
	Then the proof and the conclusion follow by proceed verbatim as in part~(ix) in~\cite{CiaMarPie2018}, with the sole difference that, instead of the ``density property'', we resort to Lemma~\ref{lem6}.
	
	$(vii)$ \emph{The weak convergences}
	\begin{equation*}
		\begin{aligned}
			u_\alpha(\varepsilon) \rightharpoonup u_\alpha&,\quad\textup{ in }H^1(\Omega) \textup{ as }\varepsilon\to 0,\\
			u_3(\varepsilon) \rightharpoonup u_3&,\quad\textup{ in }L^2(\Omega) \textup{ as }\varepsilon\to 0,
		\end{aligned}
	\end{equation*}
	\emph{established in} $(i)$ \emph{hold in fact for the whole family} $\{\bm{u}(\varepsilon)\}_{\varepsilon>0}$ \emph{and are, in fact, strong}. The proof follows verbatim the proof in parts~(x)--(xii) of~\cite{CiaMarPie2018} and is, for this reason, omitted.
\end{proof}

Note in passing that the assumption that $s(\varepsilon)>0$ in $\overline{\Omega}$ gives a sufficient condition ensuring (Lemma~\ref{lem5}) that $\min_{y\in\overline{\omega}}s(y)>0$, thus putting in the position of applying Lemma~\ref{lem6}. Recall that the assumption $s(\varepsilon)>0$ in $\overline{\Omega}$ also plays a crucial role in the proof of part $(v)$ of Theorem~\ref{th1}.

\section{Justification of Koiter's model for elliptic membrane shells subjected to the interior normal compliance contact condition}
\label{secKoiter}

This section is devoted to completing the investigation started in~\cite{Rodri2018} for what concerns the justification of Koiter's model for elliptic membranes subjected to the interior normal compliance contact condition. In Section~\ref{sec3}, we improved the result established in~\cite{Rodri2018} concerning the retrieval of the two-dimensional limit problem, Problem~\ref{problem2}, as a result of a rigorous asymptotic analysis as $\varepsilon\to 0$. It is now left to show that the solution of Koiter's model for elliptic membranes subjected to the interior normal compliance contact condition converges, in a suitable sense, to the solution of Problem~\ref{problem2}. Before establishing this fact, let us briefly recall how Koiter's model is formulated.

A commonly used \emph{two-dimensional} set of equations for modelling such a shell (``two-dimensional'' in the sense that it is posed over $\omega$ instead of $\Omega^\varepsilon$) was proposed in 1970 by Koiter~\cite{Koiter1970}. We now describe the modern formulation of this model in a general framework. By $\gamma_0$ we denote a non-zero measure portion of the boundary $\gamma:=\partial\omega$.
First, define the space
$$
\bm{V}_K (\omega):= \{\bm{\eta}=(\eta_i) \in H^1(\omega)\times H^1(\omega)\times H^2(\omega);\eta_i=\partial_{\nu}\eta_3=0 \textup{ on }\gamma_0\},
$$
where the symbol $\partial_{\nu}$ denotes the \emph{outer unit normal derivative operator along} $\gamma$, and define the norm $\|\cdot\|_{\bm{V}_K (\omega)}$ by
$$
\|\bm{\eta}\|_{\bm{V}_K(\omega)}:=\left\{\sum_{\alpha}\|\eta_\alpha\|_{H^1(\omega)}^2+\|\eta_3\|_{H^2(\omega)}^2\right\}^{1/2},\quad\textup{ for each }\bm{\eta}=(\eta_i)\in \bm{V}_K(\omega).
$$

Finally, define the bilinear forms $B_M(\cdot,\cdot)$ and $B_F(\cdot,\cdot)$ by
\begin{align*}
	B_M(\bm{\xi}, \bm{\eta})&:=\int_\omega a^{\alpha \beta \sigma \tau} \gamma_{\sigma\tau}(\bm{\xi}) \gamma_{\alpha\beta}(\bm{\eta}) \sqrt{a} \dd y,\\
	B_F(\bm{\xi}, \bm{\eta})&:=\dfrac13 \int_\omega a^{\alpha \beta \sigma \tau} \rho_{\sigma \tau}(\bm{\xi}) \rho_{\alpha \beta}(\bm{\eta}) \sqrt{a} \dd y,
\end{align*}
for each $\bm{\xi}=(\xi_i)\in \bm{V}_K(\omega)$ and each $\bm{\eta}=(\eta_i) \in \bm{V}_K(\omega)$.
Define the linear form $\ell^\varepsilon$ by
$$
\ell^\varepsilon(\bm{\eta}):=\int_\omega p^{i,\varepsilon} \eta_i \sqrt{a} \dd y, \textup{ for each } \bm{\eta}=(\eta_i) \in \bm{V}_K(\omega),
$$
where $p^{i,\varepsilon}(y):=\int_{-\varepsilon} ^\varepsilon f^{i,\varepsilon}(y,x_3) \dd x_3$ for a.a. $y \in \omega$.

Then the \emph{total energy} of the shell is the \emph{quadratic functional} $J:\bm{V}_K(\omega) \to \mathbb{R}$ defined by:
$$
J(\bm{\eta}):=\dfrac{\varepsilon}{2}B_M(\bm{\eta},\bm{\eta})+\dfrac{\varepsilon^3}{2}B_F(\bm{\eta},\bm{\eta})-\ell^\varepsilon(\bm{\eta}),\quad\textup{ for each }\bm{\eta} \in \bm{V}_K(\omega).
$$
The term $\frac{\varepsilon}{2}B_M(\cdot,\cdot)$ and $\frac{\varepsilon^3}{2}B_F(\cdot,\cdot)$ respectively represent the \emph{membrane part} and the \emph{flexural part} of the total energy, as aptly recalled by the subscripts ``$M$'' and ``$F$''.

Since the elliptic membrane under consideration is subjected to the interior normal compliance contact condition, the total energy of the shell remains \emph{unchanged} while the set over which the energy is to be minimised is now a \emph{strict subset of} $\bm{V}_K(\omega)$ (denoted by $\bm{U}_K(\omega)$ below), and takes into account the imposed confinement condition. These assumptions lead to the following definition of a variational problem, denoted $\mathcal{P}_K^\varepsilon(\omega)$, which constitutes \emph{Koiter's model for a general linearly elastic shell subjected to the interior normal compliance contact condition}.

\begin{customprob}{$\mathcal{P}_K^\varepsilon(\omega)$}
	\label{problem3}
	Find $\bm{\zeta}_K^\varepsilon=(\zeta_{K,i}^\varepsilon) \in \bm{U}_K(\omega):=\{\bm{\eta}=(\eta_i)\in \bm{V}_K(\omega); \eta_3+s \ge 0 \textup{ for all }y \in \omega\}$ that satisfies the variational inequalities:
	\begin{equation*}
		\varepsilon B_M(\bm{\zeta}_K^\varepsilon, \bm{\eta}-\bm{\zeta}_K^\varepsilon) + \varepsilon^3 B_F(\bm{\zeta}_K^\varepsilon, \bm{\eta}-\bm{\zeta}_K^\varepsilon) \ge \ell^\varepsilon(\bm{\eta}-\bm{\zeta}_K^\varepsilon),
	\end{equation*}
	for all $\bm{\eta}=(\eta_i) \in \bm{U}_K(\omega)$.
	\bqed
\end{customprob}

Note in passing that the constraint is defined over the entire set $\overline{\omega}$ since, thanks to the Rellich-Kondra\v{s}ov theorem (cf., e.g., Theorem~6.6-3 in~\cite{PGCLNFAA}), the immersion $H^2(\omega) \hookrightarrow \mathcal{C}^0(\overline{\omega})$ is compact when $\omega\subset\mathbb{R}^2$.

Thanks to the uniform positive-definiteness of the fourth order two-dimensional elasticity tensor $\{a^{\alpha\beta\sigma\tau}\}$ (cf., e.g., Theorem~3.3-2 in~\cite{Ciarlet2000}), thanks to an inequality of Korn's type on a general surface (cf., e.g. Theorem~2.6-4 in~\cite{Ciarlet2000}) and since the set $\bm{U}_K(\omega)$ is non-empty, closed and convex, it is straightforward to observe that Problem~\ref{problem3} admits a unique solution.

The next theorem, that constitutes the second main result of this paper, aims to show that the solution $\bm{\zeta}^\varepsilon_K$ of Problem~\ref{problem3} converges, in some suitable sense, to the solution of Problem~\ref{problem2} as $\varepsilon\to 0$.

\begin{theorem}
\label{th2}
Let $\omega$ be a domain in $\mathbb{R}^2$, and let $\bm{\theta} \in \mathcal{C}^3(\overline{\omega};\mathbb{E}^3)$ be an immersion.  
Suppose the gap function $s$ belongs to $\mathcal{C}^2(\overline{\omega})$ and satisfies $s > 0$ on $\overline{\omega}$.  
Consider a family of elliptic membrane shells with thickness $2\varepsilon$ converging to zero, all sharing the same middle surface $\bm{\theta}(\overline{\omega})$. Assume there exist functions $f^i \in L^2(\Omega)$, independent of $\varepsilon$, satisfying assumption~\eqref{ass-data}, i.e.,
$$
f^{i,\varepsilon}(x^\varepsilon) = f^i(x) \quad \textup{for a.e. } x^\varepsilon \in \Omega^\varepsilon \textup{ and for each } \varepsilon > 0.
$$

For each $\varepsilon > 0$, let $\bm{\zeta}_K^\varepsilon$ denote the solution to Problem~\ref{problem3}.  
Then the following convergences hold:
\begin{equation*}
	\begin{aligned}
		\zeta_{\alpha,K}^\varepsilon \bm{a}^\alpha \to \zeta_\alpha \bm{a}^\alpha &,\quad \textup{in } \bm{H}^1(\omega) \textup{ as } \varepsilon \to 0,\\
		\zeta_{3,K}^\varepsilon \bm{a}^3 \to \zeta_3 \bm{a}^3 &,\quad \textup{in } \bm{L}^2(\omega) \textup{ as } \varepsilon \to 0,
	\end{aligned}
\end{equation*}
where $\bm{\zeta}$ is the unique solution of Problem~\ref{problem2}.
\end{theorem}

\begin{proof}
For clarity, the proof is divided into three parts, labelled $(i)$--$(iii)$.

$(i)$ \emph{Uniform boundedness of the family $\{\bm{\zeta}_K^\varepsilon\}_{\varepsilon>0}$}.  
By assumption~\eqref{ass-data}, the variational inequalities in Problem~\ref{problem3} simplify to
$$
B_M(\bm{\zeta}_K^\varepsilon, \bm{\eta} - \bm{\zeta}_K^\varepsilon) + \varepsilon^2 B_F(\bm{\zeta}_K^\varepsilon, \bm{\eta} - \bm{\zeta}_K^\varepsilon) \geq \int_\omega p^i (\eta_i - \zeta_{K,i}^\varepsilon) \sqrt{a} \dd y,
$$
for all $\bm{\eta} \in \bm{U}_K(\omega)$.
This implies
\begin{equation*}
	B_M(\bm{\zeta}_K^\varepsilon, \bm{\zeta}_K^\varepsilon) + \varepsilon^2 B_F(\bm{\zeta}_K^\varepsilon, \bm{\zeta}_K^\varepsilon) \leq B_M(\bm{\zeta}_K^\varepsilon, \bm{\eta}) + \varepsilon^2 B_F(\bm{\zeta}_K^\varepsilon, \bm{\eta}) - \int_\omega p^i (\eta_i - \zeta_{K,i}^\varepsilon) \sqrt{a} \dd y,
\end{equation*}
for all $\bm{\eta} \in \bm{U}_K(\omega)$.
Using the uniform positive-definiteness of the tensor $\{a^{\alpha\beta\sigma\tau}\}$ (cf. Theorem~3.3-2 in~\cite{Ciarlet2000}) and a Korn-type inequality on surfaces (cf. Theorem~2.6-3 in~\cite{Ciarlet2000}), there exists a constant $C_1 > 0$ independent of $\varepsilon$ such that:
\begin{equation}
	\label{pezzo1}
	\|\bm{\zeta}_K^\varepsilon\|_{\bm{V}_M(\omega)}^2 \leq C_1 B_M(\bm{\zeta}_K^\varepsilon, \bm{\zeta}_K^\varepsilon).
\end{equation}

Furthermore, the continuity of $B_M(\cdot, \cdot)$ and $B_F(\cdot, \cdot)$ ensures the existence of a constant $C_2 > 0$ satisfying
\begin{equation}
	\label{pezzo2}
	\begin{aligned}
		&B_M(\bm{\zeta}_K^\varepsilon, \bm{\eta}) + \varepsilon^2 B_F(\bm{\zeta}_K^\varepsilon, \bm{\eta}) - \int_\omega p^i (\eta_i - \zeta_{K,i}^\varepsilon) \sqrt{a} \dd y\\
		&\leq C_2 \left( \|\bm{\zeta}_K^\varepsilon\|_{\bm{V}_M(\omega)} \|\bm{\eta}\|_{\bm{V}_M(\omega)} + \varepsilon^2 \|\bm{\zeta}_K^\varepsilon\|_{\bm{V}_M(\omega)} \|\bm{\eta}\|_{\bm{V}_M(\omega)} + \|\bm{\zeta}_K^\varepsilon\|_{\bm{V}_M(\omega)} + \|\bm{\eta}\|_{\bm{V}_M(\omega)} \right),
	\end{aligned}
\end{equation}
for all $\bm{\eta} = (\eta_i) \in \bm{U}_K(\omega)$. Setting $\bm{\eta} = \bm{0}$ in~\eqref{pezzo2} and using~\eqref{pezzo1} yields:
\begin{equation*}
	\|\bm{\zeta}_K^\varepsilon\|_{\bm{V}_M(\omega)} \leq C_1 C_2, \quad \textup{for all } \varepsilon > 0.
\end{equation*}

$(ii)$ \emph{Weak convergence of the family $\{\bm{\zeta}_K^\varepsilon\}_{\varepsilon>0}$}.  
By part $(i)$, the family $\{\bm{\zeta}_K^\varepsilon\}_{\varepsilon>0}$ is bounded in $\bm{V}_M(\omega)$.  
Hence, there exists a subsequence (still denoted $\{\bm{\zeta}_K^\varepsilon\}_{\varepsilon>0}$), a vector field $\bm{\zeta}^\ast \in \bm{V}_M(\omega)$, and functions $\rho_{\alpha\beta}^{-1} \in L^2(\omega)$ such that:
\begin{align*}
	\bm{\zeta}_K^\varepsilon \rightharpoonup \bm{\zeta}^\ast &,\quad \textup{in } \bm{V}_M(\omega) \textup{ as }\varepsilon\to 0,\\
	\varepsilon \rho_{\alpha\beta}(\bm{\zeta}_K^\varepsilon) \rightharpoonup \rho_{\alpha\beta}^{-1} &,\quad \textup{in } L^2(\omega) \textup{ as }\varepsilon\to 0,
\end{align*}
where the second convergence follows from the uniform positive-definiteness of $\{a^{\alpha\beta\sigma\tau}\}$.

Since $\bm{U}_M(\omega)$ is non-empty, closed, and convex, then it is \emph{weakly closed} (cf., e.g., \cite{Brez11}). Since $\bm{\zeta}^\varepsilon_K\in\bm{U}_M(\omega)$ for all $\varepsilon>0$, it thus follows that $\bm{\zeta}^\ast \in \bm{U}_M(\omega)$.
Fix $\bm{\eta} \in \bm{U}_K(\omega)$. The variational inequalities for $\bm{\zeta}_K^\varepsilon$ yield:
\begin{equation*}
	B_M(\bm{\zeta}_K^\varepsilon, \bm{\zeta}_K^\varepsilon) \le B_M(\bm{\zeta}_K^\varepsilon, \bm{\eta}) + \varepsilon^2 B_F(\bm{\zeta}_K^\varepsilon, \bm{\eta}) - \int_\omega p^i (\eta_i - \zeta_{K,i}^\varepsilon) \sqrt{a} \dd y.
\end{equation*}

Taking the limit as $\varepsilon \to 0$ gives:
\begin{equation*}
	\limsup_{\varepsilon \to 0} B_M(\bm{\zeta}_K^\varepsilon, \bm{\zeta}_K^\varepsilon) \le B_M(\bm{\zeta}^\ast, \bm{\eta}) - \int_\omega p^i (\eta_i - \zeta_{K,i}^\varepsilon) \sqrt{a} \dd y.
\end{equation*}

On the other hand,
$$
0 \le B_M(\bm{\zeta}_K^\varepsilon - \bm{\zeta}^\ast,\bm{\zeta}_K^\varepsilon - \bm{\zeta}^\ast) = B_M(\bm{\zeta}_K^\varepsilon,\bm{\zeta}_K^\varepsilon) - 2 B_M(\bm{\zeta}_K^\varepsilon,\bm{\zeta}^\ast) + B_M(\bm{\zeta}^\ast,\bm{\zeta}^\ast),
$$
which implies:
$$
2 B_M(\bm{\zeta}_K^\varepsilon, \bm{\zeta}^\ast) - B_M(\bm{\zeta}^\ast, \bm{\zeta}^\ast) \le B_M(\bm{\zeta}_K^\varepsilon,\bm{\zeta}_K^\varepsilon).
$$

Taking the limit as $\varepsilon \to 0$ yields:
\begin{equation*}
	B_M(\bm{\zeta}^\ast,\bm{\zeta}^\ast) \le \liminf_{\varepsilon\to 0} B_M(\bm{\zeta}_K^\varepsilon,\bm{\zeta}_K^\varepsilon).
\end{equation*}

Combining these results, we obtain:
\begin{equation*}
	B_M(\bm{\zeta}^\ast,\bm{\eta} - \bm{\zeta}^\ast) \ge \ell(\bm{\eta} - \bm{\zeta}^\ast), \quad \textup{for all } \bm{\eta} \in \bm{U}_K(\omega).
\end{equation*}

By Lemma~\ref{lem6}, this extends to all $\bm{\eta} \in \bm{U}_M(\omega)$, proving that $\bm{\zeta}^\ast$ is a solution of Problem~\ref{problem2}.  
Since Problem~\ref{problem2} has a unique solution, $\bm{\zeta} = \bm{\zeta}^\ast$.  
Consequently, the entire family $\{\bm{\zeta}_K^\varepsilon\}_{\varepsilon>0}$ weakly converges to $\bm{\zeta}$ in $\bm{V}_M(\omega)$ as $\varepsilon \to 0$.

$(iii)$ \emph{Strong convergence of the family $\{\bm{\zeta}_K^\varepsilon\}_{\varepsilon>0}$}.  
The $\bm{V}_K(\omega)$-ellipticity of $B_M(\cdot, \cdot)$ and Lemma~\ref{lem6} yield:
\begin{align*}
	0 &\leq \frac{1}{C_1} \|\bm{\zeta}_K^\varepsilon - \bm{\zeta}\|_{\bm{V}_M(\omega)}^2 \leq B_M(\bm{\zeta}_K^\varepsilon - \bm{\zeta}, \bm{\zeta}_K^\varepsilon - \bm{\zeta})
	= B_M(\bm{\zeta}_K^\varepsilon, \bm{\zeta}_K^\varepsilon) - 2 B_M(\bm{\zeta}_K^\varepsilon, \bm{\zeta}) + B_M(\bm{\zeta}, \bm{\zeta})\\
	&\le B_M(\bm{\zeta}_K^\varepsilon,\bm{\eta}) + \varepsilon^2 B_F(\bm{\zeta}_K^\varepsilon, \bm{\eta}) - \int_\omega p^i (\eta_i - \zeta_{K,i}^\varepsilon) \sqrt{a} \dd y - 2 B_M(\bm{\zeta}_K^\varepsilon,\bm{\zeta}) + B_M(\bm{\zeta},\bm{\zeta}),
\end{align*}
for all $\bm{\eta} \in \bm{U}_M(\omega)$.  
Setting $\bm{\eta} = \bm{\zeta}$ and taking the $\limsup$ as $\varepsilon \to 0$ gives:
$$
\limsup_{\varepsilon \to 0} \|\bm{\zeta}_K^\varepsilon - \bm{\zeta}\|_{\bm{V}_M(\omega)}^2 \le 0,
$$
which implies:
$$
\bm{\zeta}_K^\varepsilon \to \bm{\zeta}, \quad \textup{ in } \bm{V}_M(\omega) \textup{ as }\varepsilon\to 0,
$$
completing the proof.
\end{proof}

To complete this section, we de-scale the solution $\bm{u}(\varepsilon)$ of Problem~\ref{problem1}, and we return to the original framework of Problem~\ref{problem0}. In light of Theorems~\ref{th1} and~\ref{th2}, we are able to establish that the solution of Koiter's model for elliptic membranes subjected to the interior normal compliance contact condition (Problem~\ref{problem3}) asymptotically behaves like the solution of Problem~\ref{problem0} as the thickness $\varepsilon$ becomes smaller and smaller.

\begin{theorem}
\label{thj}
Let $\omega$ be a domain in $\mathbb{R}^2$ and let $\bm{\theta} \in \mathcal{C}^3(\overline{\omega};\mathbb{E}^3)$ be an immersion. Consider a family of elliptic membrane shells with thickness $2\varepsilon$ approaching zero and with each having the same middle surface $\bm{\theta}(\overline{\omega})$, and assume that there exist functions $f^i \in L^2(\Omega)$ independent of $\varepsilon$ such that the following assumption on the applied body force density holds:
$$
f^{i, \varepsilon} (x^\varepsilon) = f^i(x),\quad \textup{ for a.e. } x^\varepsilon \in \Omega^\varepsilon \textup{ for each }\varepsilon >0.
$$

Assume that the scaled gap function $s(\varepsilon)$ is such that $s(\varepsilon)>0$ in $\overline{\Omega}$.

Let $\bm{\zeta} \in \bm{U}_M(\omega)$ denote the solution of Problem~\ref{problem2}, and for each $\varepsilon>0$, let $\bm{u}^\varepsilon=(u_i^\varepsilon)$ denote the solution of Problem~\ref{problem1}. Then the following convergences hold:
\begin{align*}
	\dfrac{1}{2\varepsilon}\int_{-\varepsilon}^{\varepsilon} u_\alpha^\varepsilon \bm{g}^{\alpha,\varepsilon} dx_3^\varepsilon \to \zeta_\alpha \bm{a}^\alpha&,\quad\textup{ in } \bm{H}^1(\omega)\textup{ as }\varepsilon \to 0,\\ 
	\dfrac{1}{2 \varepsilon}\int_{-\varepsilon}^{\varepsilon} u_3^\varepsilon \bm{g}^{3,\varepsilon} dx_3^\varepsilon \to \zeta_3 \bm{a}^3&,\quad\textup{ in } \bm{L}^2(\omega)\textup{ as }\varepsilon \to 0.
\end{align*}
\end{theorem}
\begin{proof}
	The proof is analogous to that of Theorem~4.6-1 in~\cite{Ciarlet2000} and for this reason is omitted.
\end{proof}

A comparison with the convergences
\begin{align*}
&\zeta_{\alpha,K}^\varepsilon \bm{a}^\alpha \to \zeta_\alpha \bm{a}^\alpha,\quad\textup{ in } \bm{H}^1(\omega)\textup{ as }\varepsilon \to 0,\\
&\zeta_{3,K}^\varepsilon \bm{a}^3 \to \zeta_3 \bm{a}^3,\quad\textup{ in } \bm{L}^2(\omega)\textup{ as }\varepsilon \to 0,
\end{align*}
established in Theorem~\ref{th2} thus confirm that the solution of Problem~\ref{problem0} and the solution of Problem~\ref{problem3} exhibit the same limit behaviour as $\varepsilon \to 0$. This observation then fully justifies the formulation of our proposed Koiter's model for elliptic membranes subjected to the interior normal compliance contact condition.

\section{Augmentation of the regularity of the solution of Problem~\ref{problem2} up to the boundary}
\label{sec4}

This section is devoted to bringing to fruition the second objective of this paper, namely, establishing the higher regularity of the solution of Problem~\ref{problem2} up to the boundary.  
We recall once again that the main difficulty amounts to coping with the fact that the transverse component $\zeta_3$ of the solution is, in general, only of class $L^2(\omega)$. This fact prevents us, in general, from applying the theory of Agmon, Douglis and Nirenberg~\cite{AgmDouNir1959,AgmDouNir1964} to Problem~\ref{problem2} \emph{directly}. To address this problem, we implement the following strategy.

First, following~\cite{Pie-2022-interior}, we improve the regularity of the solution of Problem~\ref{problem2} in the interior of $\omega$.  
Second, we penalise Problem~\ref{problem2} and we write down the corresponding equations in the sense of distributions.  
Third, we recover a closed formula for the transverse component of the solution of the penalised problem, resorting to a technique originally developed by Ciarlet \& Sanchez-Palencia~\cite{CiaSanPan1996}. Finally, exploiting the properties of the gap function and the intrinsic assumption according to which the admissible displacements are infinitesimal, we recover the sought augmentation of regularity up to the boundary and we show that the transverse component of the solution of Problem~\ref{problem2} has vanishing trace along the boundary $\gamma$ of $\omega$. Let $\omega_0\subset \omega$ and $\omega_1 \subset \omega$ be such that:
\begin{equation}
	\label{sets}
	\omega_1 \subset\subset \omega_0 \subset\subset \omega.
\end{equation}

By the definition of the symbol $\subset\subset$ in~\eqref{sets}, we obtain that the quantity
\begin{equation}
	\label{d}
	d:=\dfrac{1}{2}\min\{\textup{dist}(\partial\omega_1,\partial\omega_0),\textup{dist}(\partial\omega_0,\gamma)\}
\end{equation}
is strictly greater than zero. Denote by $D_{\rho h}$ the first order (forward) finite difference quotient of either a function or a vector field in the canonical direction $\bm{e}_\rho$ of $\mathbb{R}^2$ and with increment size $0<h<d$ sufficiently small. We can regard the first order (forward) finite difference quotient of a function as a linear operator defined as follows:
$$
D_{\rho h}: L^2(\omega) \to L^2(\omega_0).
$$

The first order finite difference quotient of a function $\xi$ in the canonical direction $\bm{e}_\rho$ of $\mathbb{R}^2$ and with increment size $0<h<d$ is defined by:
$$
D_{\rho h}\xi(y):=\dfrac{\xi(y+h\bm{e}_\rho)-\xi(y)}{h},
$$
for all (or, possibly, a.a.) $y\in\omega_0$.

The first order finite difference quotient of a vector field $\bm{\xi}=(\xi_i)$ in the canonical direction $\bm{e}_\rho$ of $\mathbb{R}^2$ and with increment size $0<h<d$ is defined by
$$
D_{\rho h}\bm{\xi}(y):=\dfrac{\bm{\xi}(y+h\bm{e}_\rho)-\bm{\xi}(y)}{h},
$$
or, equivalently,
$$
D_{\rho h}\bm{\xi}(y)=(D_{\rho h}\xi_i(y)),
$$
for all (or, possibly, a.a.) $y \in \omega_0$.

Similarly, we can show that the first order (forward) finite difference quotient of a vector field is a linear operator from $\bm{L}^2(\omega)$ to $\bm{L}^2(\omega_0)$.

We define the second order finite difference quotient of a function $\xi$ in the canonical direction $\bm{e}_\rho$ of $\mathbb{R}^2$ and with increment size $0<h<d$ by
$$
\delta_{\rho h}\xi(y):=\dfrac{\xi(y+h \bm{e}_\rho)-2 \xi(y)+\xi(y-h \bm{e}_\rho)}{h^2},
$$
for all (or, possibly, a.a.) $y\in\omega_1$.

The second order finite difference quotient of a vector field $\bm{\xi}=(\xi_\alpha)$ in the canonical direction $\bm{e}_r$ of $\mathbb{R}^3$ and with increment size $0<h<d$ is defined by
$$
\delta_{\rho h}\bm{\xi}(y):=\left(\dfrac{\xi_i(y+h \bm{e}_\rho)-2 \xi_i(y)+\xi_i(y-h \bm{e}_\rho)}{h^2}\right),
$$
for all (or, possibly, a.a.) $y\in\omega_1$.

Note in passing that the second order finite difference quotient of a function $\xi$ can be expressed in terms of the first order finite difference quotient via the following identity:
\begin{equation*}
	\label{ide}
	\delta_{\rho h} \xi=D_{-\rho h} D_{\rho h} \xi.
\end{equation*}

Similarly, the second order finite difference quotient of a vector field $\bm{\xi}$ can be expressed in terms of the first order finite difference quotient via the following identity:
\begin{equation*}
	\label{ide2}
	\delta_{\rho h} \bm{\xi}=D_{-\rho h} D_{\rho h} \bm{\xi}.
\end{equation*}

Moreover, the following identities can be easily checked out (cf.\,\cite{Frehse1971} and~\cite{Pie2020-1}):
\begin{equation}
	\label{delta+}
	\begin{aligned}
		D_{\rho h}(v w)&=(E_{\rho h} w) (D_{\rho h} v)+v D_{\rho h} w,\\
		D_{-\rho h}(v w)&=(E_{-\rho h} w )(D_{-\rho h} v)+v D_{-\rho h} w,\\
		\delta_{\rho h}(vw)&=w \delta_{\rho h} v+(D_{\rho h}w )(D_{\rho h} v) +(D_{-\rho h}w)( D_{-\rho h} v)+v\delta_{\rho h}w.
	\end{aligned}
\end{equation}

Before presenting the main result of this section, we outline the key steps for proving that the solution $\bm{\zeta}$ of Problem~\ref{problem2} belongs to $H^2 \times H^2 \times H^1$ in a neighbourhood of $\omega$. The primary challenge lies in constructing an admissible test vector field to establish the enhanced regularity.

Our approach involves locally perturbing $\bm{\zeta}$ by a vector field with compact support in a neighbourhood of $\omega$, chosen sufficiently far from the boundary $\gamma$. To achieve this, we adapt the strategy from~\cite{Pie-2022-interior} for enhancing regularity in the interior of $\omega$, proceeding as follows:

\begin{itemize}
	\item[$(i)$] \emph{Fix a neighbourhood and a smooth cut-off function}:
	Select a neighbourhood $\mathcal{U}_1 \subset \omega$ centred at $y \in \omega$, ensuring $\mathcal{U}_1$ is sufficiently far from $\gamma$. Let $\varphi$ be a smooth function with compact support in $\mathcal{U}_1$.
	
	\item[$(ii)$] \emph{Construct the perturbation}:  
	Multiply $\varphi$ by a coefficient $\varrho = \mathcal{O}(h^2) > 0$ and a second-order finite difference quotient. The coefficient $\varrho$ is chosen so as to let the perturbation here constructed satisfies the geometrical constraint. The argument of the difference quotient is $\varphi$ multiplied by a smooth approximation of $\bm{\zeta}$.
	
	\item[$(iii)$] \emph{Apply Korn's inequality and integration by parts}:
	Using Theorem~\ref{korn} and classical integration-by-parts formulas for finite difference quotients (cf.~\cite{Evans2010}), we show that $\|D_{\rho h}(\varphi \bm{\zeta})\|_{H^1(\mathcal{U}_1) \times H^1(\mathcal{U}_1) \times L^2(\mathcal{U}_1)}$ is bounded independently of $h$. The fact that $\textup{supp}\varphi \subset\subset\mathcal{U}_1$ ensures that the result holds in $\mathcal{U}_1$, i.e., locally in $\omega$.
\end{itemize}

To begin with, we establish two abstract preparatory lemmas.

\begin{lemma}
	\label{lem1}
	Let $\omega$, $\omega_0$ and $\omega_1$ be subsets of $\mathbb{R}^2$ as in~\eqref{sets}.
	Assume that $\tilde{s} \in \mathcal{C}^2(\overline{\omega})$ is concave on $\omega_0$. 
	Let $\bm{\eta}=(\eta_i) \in H^1(\omega) \times H^1(\omega) \times L^2(\omega)$ be such that $\tilde{s}(y)+\eta_3(y) \ge 0$ for a.a. $y \in \omega$.
	Let $\varphi_1\in\mathcal{D}(\omega)$ be such that $0\le \varphi_1\le 1$ in $\overline{\omega}$ and such that $\textup{supp }\varphi_1 \subset\subset \omega_1$.
	
	Then, for each $0<h<d$ and all $0<\varrho<h^2/2$, the vector field $\bm{\eta}_\varrho=(\eta_{\varrho,i}) \in H^1(\omega) \times H^1(\omega) \times L^2(\omega)$ defined in a way such that
	\begin{equation*}
		\begin{aligned}
			\eta_{\varrho,\alpha}&:=\eta_\alpha,\\
			\eta_{\varrho,3}&:=\eta_3+\varrho \varphi_1 \delta_{\rho h}\eta_3,
		\end{aligned}
	\end{equation*}
	is such that $\tilde{s}(y)+\eta_{\varrho,3}(y)\ge 0$ for a.a. $y \in \omega_1$.
\end{lemma}
\begin{proof}
	For a.a. $y \in \omega_1$ we have that
	\begin{align*}
		&\eta_{\varrho,3}(y)=\eta_3(y)+\varrho \varphi_1(y) \dfrac{\eta_3(y+h \bm{e}_\rho)-2 \eta_3(y)+\eta_3(y-h \bm{e}_\rho)}{h^2}\\
		&=\dfrac{\varrho}{h^2}\varphi_1(y)\eta_3(y+h\bm{e}_\rho)+\left(1-\dfrac{2\varrho}{h^2}\varphi_1(y)\right) \eta_3(y)+\dfrac{\varrho}{h^2} \varphi_1(y) \eta_3(y-h\bm{e}_\rho).
	\end{align*}
	
	By virtue of the properties of $\varphi_1$ and $\varrho$, the functions $c_1(y)=c_{-1}(y):=\varrho h^{-2}\varphi_1(y)$ and $c_0(y):=1-2\varrho h^{-2}\varphi_1(y)$ are non-negative in $\overline{\omega}$.
	Combining these properties with the assumption that $\tilde{s}(y)+\eta_3(y)\ge 0$ for a.a. $y\in\omega$ gives:
	\begin{equation*}
		\begin{aligned}
			\eta_{\varrho,3}(y)&\ge -c_1(y) \tilde{s}(y+h \bm{e}_\rho)-c_0(y) \tilde{s}(y) -c_{-1}(y) \tilde{s}(y-h \bm{e}_\rho)\\
			&=-\left(\tilde{s}(y)+\varrho\varphi_1(x) \dfrac{\tilde{s}(y+h \bm{e}_\rho)-2 \tilde{s}(y)+\tilde{s}(y-h \bm{e}_\rho)}{h^2}\right),
		\end{aligned}
	\end{equation*}
	for a.a. $y\in \omega_1$.
	The assumed concavity of $\tilde{s}$ in $\omega_0$ gives
	$$
	\delta_{\rho h}\tilde{s}(y)=\dfrac{\tilde{s}(y+h \bm{e}_\rho)-2 \tilde{s}(y)+\tilde{s}(y-h \bm{e}_\rho)}{h^2}<0,
	$$
	for all $0<h<d$ and all $y \in \omega_1$. Recalling that $\textup{supp }\varphi_1 \subset\subset \omega_1$, we derive $\tilde{s}(y)+\eta_{\varrho,3}(y)\ge 0$ for a.a. $y\in\omega_1$, as it was to be proved.
\end{proof}

We observe that the concavity of the function $\tilde{s}$ has to be evaluated in $\omega_0$ since the points of the form $(y\pm h\bm{e}_\rho)$, with $y \in\omega_1$, may lie outside $\overline{\omega_1}$.
For treating the case where the concavity assumption does not hold, we need the following auxiliary result, whose proof is inspired by the original one in~\cite{Frehse1971}.

\begin{lemma}
	\label{lem2}
	Let the function $\tilde{s}\in\mathcal{C}^2(\overline{\omega})$ be such that $\tilde{s}>0$ in $\overline{\omega}$.
	Then, for every $y_0=(y_{0,1},y_{0,2}) \in \omega$, there exists a neighbourhood $\mathcal{U}$ of $y_0$ and numbers $B \in \mathbb{R}$, $B_0>0, r>0$ such that the function $(-\tilde{s}+B)\tilde{g}$ is convex in $\mathcal{U}$, where:
	$$
	\tilde{g}(y):=1-\dfrac{1}{2}\prod_{\rho=1}^2 \exp(r y_\rho-r y_{0,\rho}),\quad\textup{ for all } y=(y_1,y_2)\in\overline{\omega}.
	$$
	
	Moreover, it results $\tilde{g}(y) \ge B_0$, for all $y\in\mathcal{U}$.
\end{lemma}
\begin{proof}
	Fix $y_0 \in \omega$. Owing to the fact that $\tilde{s}\in \mathcal{C}^2(\overline{\omega})$ and $\tilde{s}>0$ in $\overline{\omega}$, we can find numbers $B\in \mathbb{R}$ and $T>0$, and a neighbourhood $\mathcal{U}_0$ of $y_0$ such that:
	$$
	-\tilde{s}(y)+B \le -T <0, \quad\textup{ for all } y \in\mathcal{U}_0.
	$$
	
	For all $y=(y_1,y_2) \in \overline{\omega}$, define the function:
	$$
	\Pi(y):=\prod_{\rho=1}^2 \exp(r y_\rho-ry_{0,\rho}).
	$$
	
	We recall that a function is convex if and only if it is convex along any lines that intersect the function domain (cf., e.g., page~67 of~\cite{Boyd2004}). In other words, checking the convexity of the function $(-\tilde{s}+B)\tilde{g}$ in $\mathcal{U}$ amounts to checking that, for all $\bm{v}=(v_1,v_2) \in \mathbb{R}^2$ and all $y=(y_1,y_2)\in\mathcal{U}$, the function
	$$
	H(t):=(-\tilde{s}(y+t\bm{v})+B)\tilde{g}(y+t\bm{v}),\quad t\in\mathbb{R}, (y+t\bm{v})=(y_1+tv_1,y_2+tv_2) \in \mathcal{U},
	$$
	is convex. Let us fix an arbitrary point $y \in \mathcal{U}$, a vector $\bm{v}=(v_1,v_2) \in \mathbb{R}^2$ and a scalar $t$ with the aforementioned properties. By direct computation, we have that
	\begin{equation*}
		\dfrac{\dd H}{\dd t}(t)=-\left((\nabla \tilde{s}(y+t\bm{v}))^T \bm{v}\right) \tilde{g}(x+t\bm{v})+\left(-\tilde{s}(y+t\bm{v})+B\right)\dfrac{\dd}{\dd t}\left(\tilde{g}(y+t\bm{v})\right),
	\end{equation*}
	and that
	\begin{align*}
		\dfrac{\dd^2 H}{\dd t^2}(t)&=\left(-\bm{v}^T (\nabla^2 \tilde{s}(y+t\bm{v})) \bm{v}\right)\tilde{g}(x+t\bm{v})+2\left(-\nabla \tilde{s}(y+t\bm{v}) \cdot \bm{v}\right) \dfrac{\dd}{\dd t}\left(\tilde{g}(y+t\bm{v})\right)\\
		&\quad+(-\tilde{s}(y+t\bm{v})+B)\dfrac{\dd^2}{\dd t^2}\left(\tilde{g}(y+t\bm{v})\right).
	\end{align*}
	
	By the definition of $\tilde{g}$, we have that
	$$
	\dfrac{\dd \tilde{g}}{\dd t}(y+t\bm{v})=\nabla \tilde{g}(y+t\bm{v}) \cdot \bm{v}=-\dfrac{1}{2} \dfrac{\dd \Pi}{\dd t}(y+t\bm{v})=-\dfrac{r}{2}(v_1+v_2) \Pi(y+t\bm{v}),
	$$
	and that
	$$
	\dfrac{\dd^2 \tilde{g}}{\dd t^2}(y+t\bm{v})=\bm{v}^T \nabla^2 \tilde{g}(y+t\bm{v}) \bm{v}=-\dfrac{r^2}{2}(v_1+v_2)^2 \Pi(y+t\bm{v}).
	$$
	
	Thanks to the uniform boundedness of the first and second derivative of $\tilde{s}$ in $\overline{\omega}$, the properties of the numbers $B$ and $T$, and the positiveness of the function $\Pi$ we derive that there exists a positive number $M$ such that:
	$$
	\dfrac{1}{\Pi}\dfrac{\dd^2 H}{\dd t^2}(t) \ge -M \left|\dfrac{1}{\Pi}-\dfrac{1}{2}\right|-rM+\dfrac{T}{2}(v_1+v_2)^2r^2.
	$$
	
	We observe that, for $r$ sufficiently large,
	$$
	p_1(r):=\dfrac{T}{2}(v_1+v_2)^2 r^2-Mr-M>0.
	$$
	
	Let us choose a neighbourhood $\mathcal{U}$ of $y_0$ such that $\mathcal{U} \subset\subset \mathcal{U}_0$ and 
	\begin{equation}
		\label{pi0}
		|y_\rho-y_{0,\rho}|\le \dfrac{1}{2r}\ln\dfrac{3}{2},\quad\textup{ for all }y \in \mathcal{U} \textup{ and all }\rho\in\{1,2\}.
	\end{equation}
	
	On the one hand, it is immediate to see that, by the monotonicity of the exponential and~\eqref{pi0},
	$$
	0<\Pi(y)\le \prod_{\rho=1}^2 \exp\left(\ln\sqrt{\dfrac{3}{2}}\right)=\dfrac{3}{2},\quad\textup{ for all }y \in \mathcal{U},
	$$
	so that we have
	\begin{equation}
		\label{pi1}
		0<\Pi(y) \le \dfrac{3}{2}, \quad\textup{ for all }y \in \mathcal{U}.
	\end{equation}
	
	On the other hand, again by~\eqref{pi0}, it results
	\begin{equation}
		\label{pi2}
		\dfrac{1}{\Pi(y)}-\dfrac{1}{2}=\prod_{\rho=1}^2\exp(-r(y_\rho-y_{0,\rho}))-\dfrac{1}{2}\le 1, \quad\textup{ for all }y \in \mathcal{U},
	\end{equation}
	and, by virtue of~\eqref{pi1}:
	\begin{equation}
		\label{pi3}
		\dfrac{1}{\Pi(y)}-\dfrac{1}{2}\ge \dfrac{2}{3}-\dfrac{1}{2}=\dfrac{1}{6}>0, \quad\textup{ for all }y \in \mathcal{U}.
	\end{equation}
	
	In conclusion, putting~\eqref{pi2} and~\eqref{pi3} together, we have that
	\begin{equation}
		\label{pi4}
		\dfrac{1}{\Pi(y)}-\dfrac{1}{2}=\left|\dfrac{1}{\Pi(y)}-\dfrac{1}{2}\right| \le 1, \quad\textup{ for all }y \in \mathcal{U}.
	\end{equation}
	
	An application of~\eqref{pi4} immediately gives that, for $r$ sufficiently large:
	$$
	\dfrac{1}{\Pi}\dfrac{\dd^2 H}{\dd t^2}(t)\ge -M \left|\dfrac{1}{\Pi}-\dfrac{1}{2}\right|-Mr+\dfrac{T}{2}(v_1+v_2)^2r^2\ge p_1(r)\ge 0.
	$$
	
	By virtue of~\eqref{pi1}, we also obtain $\tilde{g}(y)= 1-\Pi/2\ge 1/4$, for all $x \in \mathcal{U}$. Letting $B_0:=1/4$ completes the proof.
\end{proof}

We observe that the assumption $\tilde{s}>0$ in $\overline{\Omega}$ in Lemma~\ref{lem2} is related to the result stated in Lemma~\ref{lem6}.

We are now ready to prove the third main result of this paper, namely, the augmented regularity of the solution $\bm{\zeta}$ of Problem~\ref{problem2} in the interior of $\omega$.

\begin{theorem}
	\label{th3}
	Assume that the vector field $\bm{f}^\varepsilon=(f^{i,\varepsilon})$ defining the applied body force density satisfies the assumptions on the data~\eqref{ass-data}, and that it is also of class $L^2(\Omega^\varepsilon) \times L^2(\Omega^\varepsilon) \times H^1(\Omega^\varepsilon)$.
	Then, the solution $\bm{\zeta}$ of Problem~\ref{problem2} is also of class $H^2_{\textup{loc}}(\omega) \times H^2_{\textup{loc}}(\omega) \times H^1_{\textup{loc}}(\omega)$.
\end{theorem}
\begin{proof}
	For sake of clarity we break the proof into four steps,  numbered $(i)$--$(iv)$. Let $y_0 \in \omega$ and let $\mathcal{U}$ be the neighbourhood of $y_0$ constructed in Lemma~\ref{lem2} in correspondence of the gap function $s$ which is, by assumption, of class $\mathcal{C}^2(\overline{\omega})$. 
	Let $B\in\mathbb{R}$ and the function $\tilde{g}$ be defined as in Lemma~\ref{lem2}.
	
	Fix $\varphi \in \mathcal{D}(\mathcal{U})$ such that $\textup{supp }\varphi \subset\subset \mathcal{U}_1 \subset\subset \mathcal{U}$, for some neighbourhood $\mathcal{U}_1$ of $y_0$, such that $0\le \varphi \le 1$, and such that $\varphi \equiv 1$ in a compact strict subset of its support. Following the notation of Lemma~\ref{lem1}, here we let $\omega_0:=\mathcal{U}$, and $\omega_1:=\mathcal{U}_1$. In what follows, the number $0<h < d$ is fixed, where the number $d$ has been defined in~\eqref{d}.
	
	$(i)$ \emph{Construction of a candidate test field}.
	Let $0<\varrho<h^2/2$. The vector field
	$$
	\tilde{s}:=-(-s+B)\tilde{g},
	$$	
	is of class $\mathcal{C}^2(\overline{\omega})$ and, by Lemma~\ref{lem2}, it is such that $\tilde{s}$ is concave in $\mathcal{U}$. Without loss of generality, we can assume that $\mathcal{U_1}$ is a domain; otherwise, we take a polygon $Q$ such that $\textup{supp }\varphi \subset \subset Q \subset\subset \mathcal{U}_1$ and we rename it $\mathcal{U}_1$ without any loss of rigour.
	
	Since, by Lemma~\ref{lem2}, it results $\tilde{g} \ge B_0 >0$ in $\mathcal{U}$, and since $\bm{\zeta}=(\zeta_i) \in \bm{U}_M(\omega)$, we obtain that
	\begin{equation*}
		\label{constraint1}
		(\zeta_3(y)+B)\tilde{g}(y)\ge (-s(y)+B) \tilde{g}(y)=-\tilde{s}(y), \textup{ for a.a. } y \in \omega.
	\end{equation*}
	
	Thanks to the concavity of $\tilde{s}$ in $\mathcal{U}$, letting $\varphi_1=\varphi^2$ in the notation of Lemma~\ref{lem1}, we obtain that:
	\begin{equation}
		\label{test3}
		-\tilde{s}(y) \le (\zeta_3(y)+B)\tilde{g}(y)+\varrho (\varphi(y))^2 \delta_{\rho h}\left((\zeta_3(y)+B)\tilde{g}(y)\right),
	\end{equation}
	for a.a. $y\in\mathcal{U}_1$. Dividing~\eqref{test3} by $\tilde{g}>0$ and then subtracting $B$ from each member of~\eqref{test3} gives:
	\begin{equation*}
		\label{test1}
		\begin{aligned}
			-s(y)\le \zeta_3(y)+\varrho (\tilde{g}(y))^{-1} (\varphi(y))^2 \delta_{\rho h}\left((\zeta_3(y)+B)\tilde{g}(y)\right), \quad\textup{ for a.a. } y\in\mathcal{U}_1.
		\end{aligned}
	\end{equation*}
	
	Therefore, we define the vector field $\tilde{\bm{\zeta}}_\varrho=(\tilde{\zeta}_{\varrho,j})$ as follows:
	\begin{equation*}
		\begin{aligned}
			\tilde{\zeta}_{\varrho,\alpha}&:=\zeta_\alpha+\varphi\delta_{\rho h}(\varphi\zeta_\alpha),\\
			\tilde{\zeta}_{\varrho,3}&:=\zeta_3+\varrho \dfrac{\varphi^2}{\tilde{g}} \delta_{\rho h}\left((\zeta_3+B)\tilde{g}\right).
		\end{aligned}
	\end{equation*}
	
	The vector field $\tilde{\bm{\zeta}}_\varrho=(\tilde{\zeta}_{\varrho,j})$ will be the one at which the variational inequalities in Problem~\ref{problem2} will be tested in order to establish the augmentation of regularity in the interior of $\omega$.
	
	$(ii)$ \emph{Estimates for the forcing term acting.} Given two functions $f_1, f_2 \in L^2(\mathcal{U}_1)$, we say that $f_1 \lesssim f_2$ if there exists a constant $C>0$ independent of $h$ and eventually dependent on $\|\bm{p}^\varepsilon\|_{\bm{H}^1(\omega)}$, $\tilde{g}$, its reciprocal $\tilde{g}^{-1}$, $\|\bm{\zeta}^\varepsilon\|_{H^1(\omega)\times H^1(\omega)\times L^2(\omega)}$, $\varphi$ and its derivatives such that:
	\begin{equation*}
		\label{key-relation}
		\int_{\mathcal{U}_1} f_1 \sqrt{a}\dd y \le \int_{\mathcal{U}_1} f_2 \sqrt{a}\dd y +C(1+\|D_{\rho h}(\varphi \bm{\zeta})\|_{H^1(\mathcal{U}_1) \times H^1(\mathcal{U}_1) \times L^2(\mathcal{U}_1)}).
	\end{equation*}
	
	We first show that there exists a constant $C>0$ independent of $h$ but eventually dependent on $\|\bm{p}\|_{L^2(\omega)\times L^2(\omega)\times H^1(\omega)}$, $\tilde{g}$, its reciprocal $\tilde{g}^{-1}$, $\|\bm{\zeta}\|_{H^1(\omega)\times H^1(\omega)\times L^2(\omega)}$, $\varphi$ and its derivatives such that:
	\begin{equation}
		\label{est-1}
		\begin{aligned}
			-\int_{\mathcal{U}_1} p^\alpha \varphi\delta_{\rho h}(\varphi\zeta_\alpha) \sqrt{a} \dd y&\le C(1+\|D_{\rho h}(\varphi \bm{\zeta})\|_{H^1(\mathcal{U}_1) \times H^1(\mathcal{U}_1) \times L^2(\mathcal{U}_1)}),\\
			-\int_{\mathcal{U}_1} p^3 \varrho \dfrac{\varphi^2}{\tilde{g}} \delta_{\rho h}\left((\zeta_3+B)\tilde{g}\right) \sqrt{a} \dd y&\le C(1+\|D_{\rho h}(\varphi \bm{\zeta})\|_{H^1(\mathcal{U}_1) \times H^1(\mathcal{U}_1) \times L^2(\mathcal{U}_1)}).
		\end{aligned}
	\end{equation}
	
	The first estimate in~\eqref{est-1} is a straightforward application of the integration-by-parts formula for finite difference quotients (cf., e.g., Section~5.8.2 in~\cite{Evans2010}). For the second estimate, note that an application of formula~\eqref{delta+} gives:
	\begin{equation}
		\label{term-1}
		\delta_{\rho h}(\zeta_3 \tilde{g})=\tilde{g}\delta_{\rho h}\zeta_3+(D_{\rho h} \tilde{g}) (D_{\rho h} \zeta_3)+(D_{-\rho h} \tilde{g}) (D_{-\rho h} \zeta_3)+\zeta_3 \delta_{\rho h}\tilde{g}.
	\end{equation}
	
	Consequently, we obtain the following estimates for the first term in~\eqref{term-1}:
	\begin{equation}
		\label{term-2}
		\varphi\delta_{\rho h}\zeta_3=\delta_{\rho h}(\varphi \zeta_3)-(D_{\rho h} \varphi) (D_{\rho h} \zeta_3)(D_{-\rho h} \varphi) (D_{-\rho h} \zeta_3)-\zeta_3 \delta_{\rho h}\varphi.
	\end{equation}
	
	Note in passing that only the first term on the right-hand side of~\eqref{term-2} requires that $p^3\in H^1(\omega)$ in order to apply the integration-by-parts formula for finite difference quotients (cf., e.g., Section~5.8.2 in~\cite{Evans2010}).
	
	Applying~\eqref{term-1} and~\eqref{term-2} to the left-hand side of the to-be-established~\eqref{est-1} gives
	\begin{equation*}
		\begin{aligned}
			&-\int_{\mathcal{U}_1} p^3 \varrho \dfrac{\varphi^2}{\tilde{g}} \delta_{\rho h}\left((\zeta_3+B)\tilde{g}\right) \sqrt{a} \dd y
			\le -\int_{\mathcal{U}_1} p^3 \varrho \dfrac{\varphi^2}{\tilde{g}} B \delta_{\rho h}\tilde{g} \sqrt{a} \dd y -\int_{\mathcal{U}_1} p^3 \varrho \dfrac{\varphi^2}{\tilde{g}} \delta_{\rho h}(\zeta_3 \tilde{g}) \sqrt{a} \dd y\\
			& \le C-\int_{\mathcal{U}_1} p^3 \varrho \varphi \delta_{\rho h}(\varphi \zeta_3) \sqrt{a} \dd y,
		\end{aligned}
	\end{equation*}
	from which it is immediate to establish~\eqref{est-1} as a result of an application of the integration-by-parts formula for finite difference quotients (cf., e.g., Section~5.8.2 in~\cite{Evans2010}), whose applicability hinges on the assumption $p^3 \in H^1(\omega)$. By Lemma~\ref{lem2} we know that the derivatives of $\tilde{g}$ and $\tilde{g}^{-1}$ are uniformly bounded in $\mathcal{U}$. As a result, the computations above imply:
	\begin{equation*}
		\label{e1''}
		-a^{\alpha\beta\sigma\tau}\gamma_{\sigma \tau}(\bm{\zeta})\gamma_{\alpha\beta}(\tilde{\bm{\zeta}}_\varrho-\bm{\zeta})\lesssim 0.
	\end{equation*}
	
	$(iii)$ \emph{Estimates for the bulk energy term.} The next step of the proof consists in showing that:
	\begin{equation*}
		\label{est-2}
		-a^{\alpha\beta\sigma\tau}\gamma_{\sigma \tau}(\varphi\bm{\zeta})\gamma_{\alpha\beta}(\delta_{\rho h}(\varphi \bm{\zeta}))
		\lesssim -a^{\alpha\beta\sigma\tau}\gamma_{\sigma \tau}(\bm{\zeta})\gamma_{\alpha\beta}(\tilde{\bm{\zeta}}_\varrho-\bm{\zeta}).
	\end{equation*}
	
	To establish~\eqref{est-2}, we analyze the terms in the bilinear form for the elastic energy individually. For the evaluation of the nine subsequent terms, the indices are held fixed, meaning the summation convention for repeated indices is not applied in~\eqref{term-3}--\eqref{term-11} below.
	
	First, thanks to an application of Green's formula (cf., e.g., Theorem~6.6-7 of~\cite{PGCLNFAA}), we estimate:
	\begin{equation}
		\label{term-3}
		\begin{aligned}
			&\int_{\mathcal{U}_1} -a^{\alpha\beta\sigma\tau} \partial_\sigma(\varphi \zeta_\tau) \partial_\beta(\delta_{\rho h}(\varphi \zeta_\alpha)) \sqrt{a} \dd y
			=\int_{\mathcal{U}_1} -a^{\alpha\beta\sigma\tau} [(\partial_\sigma \varphi)\zeta_\tau +\varphi \partial_\sigma \zeta_\tau] \partial_\beta(\delta_{\rho h}(\varphi \zeta_\alpha)) \sqrt{a} \dd y\\
			&=\int_{\mathcal{U}_1} \partial_\beta(a^{\alpha\beta\sigma\tau} (\partial_\sigma \varphi) \zeta_\tau) \delta_{\rho h}(\varphi \zeta_\alpha) \sqrt{a} \dd y+\int_{\mathcal{U}_1} -a^{\alpha\beta\sigma\tau} \partial_\sigma \zeta_\tau [\varphi\partial_\beta(\delta_{\rho h}(\varphi \zeta_\alpha))] \sqrt{a} \dd y\\
			&\le C \|\zeta_\tau\|_{H^1(\mathcal{U}_1) \times H^1(\mathcal{U}_1) \times L^2(\mathcal{U}_1)} \|D_{\rho h}(\varphi \zeta_\alpha)\|_{H^1(\mathcal{U}_1)}+\int_{\mathcal{U}_1} -a^{\alpha\beta\sigma\tau} \partial_\sigma \zeta_\tau \partial_\beta(\varphi \delta_{\rho h}(\varphi \zeta_\alpha)) \sqrt{a} \dd y\\
			&\quad +  \int_{\mathcal{U}_1} a^{\alpha\beta\sigma\tau} (\partial_\sigma \zeta_\tau) (\partial_\beta \varphi) \delta_{\rho h}(\varphi \zeta_\alpha) \sqrt{a} \dd y\\
			&\le \int_{\mathcal{U}_1} -a^{\alpha\beta\sigma\tau} (\partial_\sigma \zeta_\tau) \partial_\beta(\varphi\delta_{\rho h}(\varphi \zeta_\alpha)) \sqrt{a} \dd y +C \|D_{\rho h}(\varphi \bm{\zeta})\|_{H^1(\mathcal{U}_1)\times H^1(\mathcal{U}_1)\times L^2(\mathcal{U}_1)}.
		\end{aligned}
	\end{equation}

Second, we estimate:
\begin{equation}
	\label{term-4}
	\begin{aligned}
		&\int_{\mathcal{U}_1} -a^{\alpha\beta\sigma\tau} (-\Gamma_{\sigma\tau}^\varsigma \varphi \zeta_\varsigma) \partial_\beta(\delta_{\rho h}(\varphi \zeta_\alpha)) \sqrt{a} \dd y
		=-\int_{\mathcal{U}_1} \partial_\beta (a^{\alpha\beta\sigma\tau} \Gamma_{\alpha\beta}^\varsigma \varphi \zeta_\varsigma) \delta_{\rho h}(\varphi \zeta_\alpha) \sqrt{a} \dd y\\
		&\le C \|\zeta_\varsigma\|_{H^1(\mathcal{U}_1)} \|D_{\rho h} (\varphi \bm{\zeta})\|_{H^1(\mathcal{U}_1)\times H^1(\mathcal{U}_1)\times L^2(\mathcal{U}_1)},
	\end{aligned}
\end{equation}
where the equality holds as a consequence of Green's formula.

Third, we estimate:
\begin{equation}
	\label{term-5}
	\begin{aligned}
		&\int_{\mathcal{U}_1} -a^{\alpha\beta\sigma\tau} (-b_{\alpha\beta}\varphi \zeta_3) \partial_\beta(\delta_{\rho h}(\varphi \zeta_\alpha)) \sqrt{a} \dd y
		=\int_{\mathcal{U}_1} a^{\alpha\beta\sigma\tau} (b_{\alpha\beta} \zeta^\varepsilon_3) \partial_\beta(\varphi \delta_{\rho h}(\varphi \zeta_\alpha)) \sqrt{a} \dd y\\
		&\quad-\int_{\mathcal{U}_1} a^{\alpha\beta\sigma\tau} b_{\alpha\beta} \zeta^\varepsilon_3 (\partial_\beta \varphi) \delta_{\rho h}(\varphi \zeta_\alpha) \sqrt{a} \dd y\\
		&\le \int_{\mathcal{U}_1} a^{\alpha\beta\sigma\tau} (b_{\alpha\beta} \zeta^\varepsilon_3) \partial_\beta(\varphi \delta_{\rho h}(\varphi \zeta_\alpha)) \sqrt{a} \dd y +C \|D_{\rho h}(\varphi \bm{\zeta})\|_{H^1(\mathcal{U}_1)\times H^1(\mathcal{U}_1)\times L^2(\mathcal{U}_1)}.
	\end{aligned}
\end{equation}

Fourth, we estimate:
\begin{equation}
	\label{term-6}
	\begin{aligned}
		&\int_{\mathcal{U}_1} -a^{\alpha\beta\sigma\tau} \partial_\sigma(\varphi \zeta_\tau) \Gamma_{\alpha\beta}^\upsilon \delta_{\rho h}(\varphi \zeta_\upsilon) \sqrt{a} \dd y
		=\int_{\mathcal{U}_1} -a^{\alpha\beta\sigma\tau} (\partial_\sigma\varphi) \zeta_\tau \Gamma_{\alpha\beta}^\upsilon \delta_{\rho h}(\varphi \zeta_\upsilon) \sqrt{a} \dd y\\
		&\quad+\int_{\mathcal{U}_1} -a^{\alpha\beta\sigma\tau} \varphi (\partial_\sigma \zeta_\tau) \Gamma_{\alpha\beta}^\upsilon \delta_{\rho h}(\varphi \zeta_\upsilon) \sqrt{a} \dd y\\
		&\le C \|D_{\rho h}(\varphi \bm{\zeta})\|_{H^1(\mathcal{U}_1)\times H^1(\mathcal{U}_1)\times L^2(\mathcal{U}_1)}
		+\int_{\mathcal{U}_1} -a^{\alpha\beta\sigma\tau} (\partial_\sigma \zeta_\tau) \Gamma_{\alpha\beta}^\upsilon [\varphi\delta_{\rho h}(\varphi \zeta_\upsilon)] \sqrt{a} \dd y.
	\end{aligned}
\end{equation}

Fifth, we straightforwardly observe that:
\begin{equation}
	\label{term-7}
	\begin{aligned}
		&\int_{\mathcal{U}_1} -a^{\alpha\beta\sigma\tau} (-\Gamma_{\sigma\tau}^\varsigma \zeta_\varsigma \varphi) \Gamma_{\alpha\beta}^\upsilon \delta_{\rho h}(\zeta_\upsilon \varphi) \sqrt{a} \dd y
		\le C(1+\|D_{\rho h}(\varphi \bm{\zeta})\|_{H^1(\mathcal{U}_1) \times H^1(\mathcal{U}_1) \times L^2(\mathcal{U}_1)})\\
		&\quad+ \int_{\mathcal{U}_1} -a^{\alpha\beta\sigma\tau} (-\Gamma_{\sigma\tau}^\varsigma \zeta_\varsigma) \Gamma_{\alpha\beta}^\upsilon [\varphi\delta_{\rho h}(\zeta_\upsilon \varphi)] \sqrt{a} \dd y.
	\end{aligned}
\end{equation}

Sixth, we straightforwardly observe that:
\begin{equation}
	\label{term-8}
	\begin{aligned}
		&\int_{\mathcal{U}_1} -a^{\alpha\beta\sigma\tau} (b_{\sigma\tau} \zeta_3 \varphi) \Gamma_{\alpha\beta}^\upsilon \delta_{\rho h}(\varphi \zeta_\upsilon) \sqrt{a} \dd y
		\le C(1+\|D_{\rho h}(\varphi \bm{\zeta})\|_{H^1(\mathcal{U}_1) \times H^1(\mathcal{U}_1) \times L^2(\mathcal{U}_1)})\\
		&\quad+\int_{\mathcal{U}_1} -a^{\alpha\beta\sigma\tau} b_{\sigma\tau} \zeta_3  \Gamma_{\alpha\beta}^\upsilon [\varphi\delta_{\rho h}(\varphi \zeta_\upsilon)] \sqrt{a} \dd y.
	\end{aligned}
\end{equation}

Seventh, we straightforwardly observe that:
\begin{equation}
	\label{term-9}
	\begin{aligned}
		&\int_{\mathcal{U}_1} -a^{\alpha\beta\sigma\tau} (-\Gamma_{\sigma\tau}^\varsigma \varphi \zeta_\varsigma) b_{\alpha\beta} \delta_{\rho h}(\zeta_3 \varphi) \sqrt{a} \dd y
		\le C(1+\|D_{\rho h}(\varphi \bm{\zeta})\|_{H^1(\mathcal{U}_1) \times H^1(\mathcal{U}_1) \times L^2(\mathcal{U}_1)})\\
		&\quad+\int_{\mathcal{U}_1} -a^{\alpha\beta\sigma\tau} (-\Gamma_{\sigma\tau}^\varsigma \zeta_\varsigma) b_{\alpha\beta} [\varphi^2\delta_{\rho h}\zeta_3] \sqrt{a} \dd y.
	\end{aligned}
\end{equation}

Eighth, we estimate:
\begin{equation}
	\label{term-10}
	\begin{aligned}
		&\int_{\mathcal{U}_1} -a^{\alpha\beta\sigma\tau} \partial_\sigma(\varphi \zeta_\tau) b_{\alpha\beta} \delta_{\rho h}(\zeta_3 \varphi) \sqrt{a} \dd y
		=\int_{\mathcal{U}_1} -a^{\alpha\beta\sigma\tau} (\partial_\sigma \zeta_\tau) b_{\alpha\beta} [\varphi \delta_{\rho h}(\zeta_3 \varphi)] \sqrt{a} \dd y\\
		&\quad+\int_{\mathcal{U}_1} -a^{\alpha\beta\sigma\tau} (\partial_\sigma \varphi) \zeta_\tau b_{\alpha\beta} \delta_{\rho h}(\zeta_3 \varphi)\sqrt{a} \dd y\\
		&=\int_{\mathcal{U}_1} -a^{\alpha\beta\sigma\tau} (\partial_\sigma \zeta_\tau) b_{\alpha\beta} [\varphi \delta_{\rho h}(\zeta_3 \varphi)] \sqrt{a} \dd y+\int_{\mathcal{U}_1} D_{\rho h}(-a^{\alpha\beta\sigma\tau} (\partial_\sigma \varphi) \zeta_\tau b_{\alpha\beta}) D_{\rho h}(\varphi \zeta_3) \sqrt{a} \dd y\\
		&\le \int_{\mathcal{U}_1} -a^{\alpha\beta\sigma\tau} (\partial_\sigma \zeta_\tau) b_{\alpha\beta} [\varphi^2 \delta_{\rho h}\zeta_3] \sqrt{a} \dd y+ C(1+\|D_{\rho h}(\varphi \bm{\zeta})\|_{H^1(\mathcal{U}_1) \times H^1(\mathcal{U}_1) \times L^2(\mathcal{U}_1)}),
	\end{aligned}
\end{equation}
where in the last equality we used the integration-by-parts formula for finite difference quotients.

Ninth, and last, we straightforwardly observe that
\begin{equation}
	\label{term-11}
	\begin{aligned}
		&\int_{\mathcal{U}_1} -a^{\alpha\beta\sigma\tau} (b_{\sigma\tau} \zeta_3 \varphi) b_{\alpha\beta} \delta_{\rho h}(\zeta_3 \varphi) \sqrt{a} \dd y
		=\int_{\mathcal{U}_1} -a^{\alpha\beta\sigma\tau} (b_{\sigma\tau} \zeta_3) b_{\alpha\beta} [\varphi\delta_{\rho h}(\zeta_3 \varphi)] \sqrt{a} \dd y\\
		&\le C(1+\|D_{\rho h}(\varphi \bm{\zeta})\|_{H^1(\mathcal{U}_1) \times H^1(\mathcal{U}_1) \times L^2(\mathcal{U}_1)})+\int_{\mathcal{U}_1} -a^{\alpha\beta\sigma\tau} (b_{\sigma\tau} \zeta_3) b_{\alpha\beta} [\varphi^2 \delta_{\rho h}\zeta_3] \sqrt{a} \dd y.
	\end{aligned}
\end{equation}

In conclusion, combining~\eqref{term-3}--\eqref{term-11} together gives~\eqref{est-2} as it was to be proved.

$(iv)$ \emph{Conclusion of the proof.} An application of the integration-by-parts formula for finite difference quotients (cf., e.g., Section~5.8.2 in~\cite{Evans2010}) gives that:
\begin{equation}
	\label{est-3}
	\int_{\mathcal{U}_1} a^{\alpha\beta\sigma\tau}\gamma_{\sigma \tau}(D_{\rho h}(\varphi \bm{\zeta}))\gamma_{\alpha\beta}(D_{\rho h}(\varphi \bm{\zeta}))  \sqrt{a} \dd y\le C(1+\|D_{\rho h}(\varphi \bm{\zeta})\|_{H^1(\mathcal{U}_1) \times H^1(\mathcal{U}_1) \times L^2(\mathcal{U}_1)}).
\end{equation}

Thanks to the uniform positive definiteness of the fourth order two-dimensional tensor $\{a^{\alpha\beta\sigma\tau}\}$ and thanks to Korn's inequality (Theorem~\ref{korn}), the left-hand side of~\eqref{est-3} can be estimates as follows
\begin{equation}
	\label{est-4}
	\|D_{\rho h}(\varphi \bm{\zeta})\|_{H^1(\mathcal{U}_1) \times H^1(\mathcal{U}_1) \times L^2(\mathcal{U}_1)}^2\le \hat{C}(1+\|D_{\rho h}(\varphi \bm{\zeta})\|_{H^1(\mathcal{U}_1) \times H^1(\mathcal{U}_1) \times L^2(\mathcal{U}_1)}),
\end{equation}
for some $\hat{C}>0$ independent of $h$. In light of~\eqref{est-4}, we are in a position to apply Theorem~3 of Section~5.8.2 of~\cite{Evans2010}. This, together with the fact that $\varphi \equiv 1$ in a compact strict subset of its support, shows that $\bm{\zeta} \in H^2_{\textup{loc}}(\omega) \times H^2_{\textup{loc}}(\omega) \times H^1_{\textup{loc}}(\omega)$, completing the proof.
\end{proof}

Finally, we establish the fourth, and final, main result of this paper, namely, the augmentation of regularity up to the boundary. We address the special case where the middle surface of the membrane shell under consideration is a lower hemisphere. This case is of interest as it constitutes an example of surface for which the sufficient conditions ensuring the validity of the ``density property'' established in~\cite{CiaMarPie2018} does not hold.

Such an augmentation of regularity result allows us to infer that the concept of trace for the transverse component of the solution of Problem~\ref{problem2} is well-defined and, moreover, that this trace vanishes along the entire boundary $\gamma$. Note that this conclusion aligns with the statement of Koiter's model (Problem~\ref{problem3}), for which the transverse component of the solution is in fact of class $H^2_0(\omega)$. In particular, the following result allows us to infer that for sufficiently smooth applied body forces, the boundary layer for the solution of Problem~\ref{problem2}, whose existence is postulated as a result of Theorems~\ref{th1} and~\ref{th2}, vanishes.

The theorem applies, in general, to those elliptic membranes for which the gap function evaluated at the boundary is \emph{large}. Combining this assumption on the gap function with the fact that, by definition of the models under consideration, the admissible displacements are intrinsically infinitesimal ensures that the constraint is inactive in a neighbourhood of $\gamma$. Note in passing that in the limit case where the membrane is a hemisphere and the obstacle is a flat plane parallel to the plane containing $\gamma$, the gap function along $\gamma$ is equal to $\infty$.

\begin{theorem}
	\label{th4}
	Assume that the boundary $\gamma$ of the Lipschitz domain $\omega\subset\mathbb{R}^2$ is of class $\mathcal{C}^2$.
	Assume that $\bm{\theta}(\overline{\omega})$ represents a lower hemisphere, so that $s=\infty$ on $\gamma$.
	Then, the solution $\bm{\zeta}$ of Problem~\ref{problem2} is of class $(H^2(\omega) \cap H^1_0(\omega)) \times (H^2(\omega) \cap H^1_0(\omega)) \times H^1_0(\omega)$ and, moreover, the same trace formula as in~\cite{CiaSanPan1996} holds.
\end{theorem}
\begin{proof}
	Since $s=\infty$ on $\gamma$, the fact that the admissible displacements are infinitesimal implies there exists a neighbourhood of the boundary where the constraint is inactive. In this neighbourhood of the boundary, we apply verbatim the technique presented in~\cite{CiaSanPan1996}, obtaining that $\zeta_3 \in H^1(\omega)$ and that $\zeta_3$ satisfies the same trace formula as in~\cite{CiaSanPan1996}. This trace formula ensures that the trace of $\zeta_3$ vanishes along $\gamma$.
\end{proof}

\section*{Conclusion and final remarks}

In this paper, we first establish the justification of Koiter's model for elliptic membrane shells subjected to an interior normal compliance contact condition that appears to be defined here for the first time. Specifically, we prove that at all (or, possibly, a.a.) points in the definition domain, the magnitude of the transverse component of the displacement is bounded above by the gap function along the normal direction at that point.

Second, we demonstrate that the solution to the limit problem, Problem~\ref{problem2}, obtained via a rigorous asymptotic analysis as the thickness $\varepsilon \to 0$ starting from either the original three-dimensional curvilinear formulation (Problem~\ref{problem1}) or the variational formulation of Koiter's model (Problem~\ref{problem3}), enjoys enhanced regularity up to the boundary. The proof of the augmentation of regularity up to the boundary relies on techniques originally developed by Ciarlet \& Sanchez-Palencia in~\cite{CiaSanPan1996}.

\section*{Declarations}

\subsection*{Authors' Contribution}

P.P.: Conceptualisation, analysis, article typesetting and proof-reading.

\subsection*{Acknowledgements}

Not applicable.

\subsection*{Ethical Approval}

Not applicable.

\subsection*{Availability of Supporting Data}

Not applicable.

\subsection*{AI Usage}

The Artificial Intelligences DeepSeek and Tencent Yuanbao were used to polish and improve the wording of this paper.

\subsection*{Competing Interests}

All authors certify that they have no affiliations with or involvement in any organisation or entity with any competing interests in the subject matter or materials discussed in this manuscript.

\subsection*{Funding}

This research was partly supported by the University Development Fund of The Chinese University of Hong Kong, Shenzhen.